\documentclass[11pt]{article}
\usepackage{afterpage}

\usepackage[utf8]{inputenc}
\usepackage{amsfonts,amssymb,amsmath,amsthm,latexsym,epsfig,amscd,bbm,stmaryrd,mathrsfs}
\usepackage[english]{babel}
\usepackage{graphicx}
\usepackage{color}
\usepackage{tikz}
\usepackage{courier}
\usepackage{bbm}
\usepackage{enumerate}
\usepackage{psfrag}
\usepackage{wasysym}
\usepackage{type1cm}
\usepackage{dsfont}
\usepackage{geometry}
\usepackage{marginnote}

\numberwithin{equation}{section}

\theoremstyle{plain} 

\newtheorem{theorem}{Theorem}[section]
\newtheorem{corollary}[theorem]{Corollary}
\newtheorem{lemma}[theorem]{Lemma}
\newtheorem{proposition}[theorem]{Proposition}
\newtheorem{definition}[theorem]{Definition}
\newtheorem{assumption}[theorem]{Assumption}

\newcommand{\ssup}[1] {{{\scriptscriptstyle{({#1}})}}} 
\newcommand{\EE}{\mathbb{E}}

\newcommand{\N}{\mathbb{N}}
\newcommand{\Z}{\mathbb{Z}}
\newcommand{\R}{\mathbb{R}}

\newcommand{\cC}{\mathcal{C}}
\newcommand{\cB}{\mathcal{B}}

\newcommand{\cO}{\mathcal{O}} 
\newcommand{\cT}{\mathcal{T}} 
\newcommand{\cZ}{\mathcal{Z}} 
\newcommand{\Prob}{\mathrm{P}}
\newcommand{\Expec}{\mathrm{E}}

\newcommand{\eee}{\mathrm{e}}

\def \cL {{\mathcal L}}

\DeclareMathOperator\supp{supp}
\DeclareMathOperator\dist{dist}
\DeclareMathOperator\diam{diam}

\newcommand{\scrP}{\mathscr{P}}

\newcommand{\scrX}{\mathscr{X}}

\newcommand{\BB}{\mathcal B}
\newcommand{\CC}{\mathcal C}

\newcommand{\MM}{\mathcal M}
\newcommand{\NN}{\mathcal N}

\newcommand{\YY}{\mathcal Y}

\def\Z{\mathbb{Z}}
\def\N{\mathbb{N}}
\def\R{\mathbb{R}}

\def\P{\mathbb{P}}
\def\E{\mathbb{E}}

\newcommand{\cc}{{\text{\rm c}}}
\newcommand{\texte}{\text{\rm e}}

\def \GW {{\mathcal{G}\mathcal{W}}}

\newcommand{\textd}{\text{\rm d}\mkern0.5mu}

\def\1{{\mathchoice {1\mskip-4mu\mathrm l}
{1\mskip-4mu\mathrm l} 
{1\mskip-4.5mu\mathrm l} {1\mskip-5mu\mathrm l}}}

\def \ba {\begin{array}}
\def \ea {\end{array}}

\def \P  {{\mathbb P}}
\def \E  {{\mathbb E}}

\def \cL {{\mathcal L}}

\def \cN {{\mathcal N}}
\def \cO {{\mathcal O}}
\def \cP {{\mathcal P}}

\def \cT {{\mathcal T}}
\def \GW {{\mathcal{G}\mathcal{W}}}

\def \e {\mathrm{e}}
\def \ee {\mathrm{e}}
\def \dd {\mathrm{d}}

\DeclareSymbolFont{symbolsC}{U}{pxsyc}{m}{n}
\DeclareMathSymbol{\opentimes}{\mathrel}{symbolsC}{93}

\newcommand{\Probgr}{\mathfrak{P}}

\title{The parabolic Anderson model\\ on a Galton-Watson tree revisited}

\author{

F.\ den Hollander
\footnotemark[1]
\\

D. Wang
\footnotemark[1]
}

\footnotetext[1]{
Mathematical Institute, Leiden University, P.O.\ Box 9512, 2300 RA Leiden, The Netherlands.
\\
Email: {\tt denholla@math.leidenuniuv.nl} ; {\tt d.wang@math.leidenuniuv.nl}}

\date{\today}

\begin{document}

\maketitle

\begin{abstract}
In \cite{dHKdS2020} a detailed analysis was given of the large-time asymptotics of the total mass of the solution to the parabolic Anderson model on a supercritical Galton-Watson random tree with an i.i.d.\ random potential whose marginal distribution is double-exponential. Under the assumption that the degree distribution has bounded support, two terms in the asymptotic expansion were identified under the quenched law, i.e., conditional on the realisation of the random tree and the random potential. The second term contains a variational formula indicating that the solution concentrates on a subtree with minimal degree according to a computable profile. The present paper extends the analysis to degree distributions with unbounded support. We identify the weakest condition on the tail of the degree distribution under which the arguments in \cite{dHKdS2020} can be pushed through. To do so we need to control the occurrence of large degrees uniformly in large subtrees of the Galton-Watson tree. 

\vspace{0.5cm}



\medskip\noindent
{\bf Acknowledgment:}
The work in this paper was supported through Gravitation-grant NETWORKS-024.002.003 of the Netherlands Organisation for Scientific Research (NWO). The authors thank G\"otz Kersting and Anton Wakolbinger for helpful discussions on large deviation properties of the Galton-Watson process.
\normalsize
\end{abstract}

\newpage


\newpage


\section{Introduction and main results}
\label{s:intro}

Section~\ref{ss:introPAM} provides a brief introduction to the parabolic Anderson model. Section~\ref{ss:PAM} introduces basic notation and key assumptions. Section~\ref{ss:GW} states the main theorem and gives an outline of the remainder of the paper. 


\subsection{The PAM and intermittency}
\label{ss:introPAM}

The \emph{parabolic Anderson model} (PAM) is the Cauchy problem
\begin{equation}
\label{PAM}
\partial_t u(x,t) = \Delta_\scrX u(x,t) + \xi(x) u(x,t) , \qquad t>0, \, x \in \scrX,
\end{equation}
where $\scrX$ is an ambient space, $\Delta_\scrX$ is a Laplace operator acting on functions on $\scrX$, and $\xi$ is a random potential on $\scrX$. Most of the literature considers the setting where $\scrX$ is either $\Z^d$ or $\R^d$ with $d \geq 1$ (for mathematical surveys we refer the reader to \cite{A2016}, \cite{K2016}). More recently, other choices for $\scrX$ have been considered as well: the complete graph \cite{FM1990}, the hypercube \cite{AGH2016}, Galton-Watson trees \cite{dHKdS2020}, and random graphs with prescribed degrees \cite{dHKdS2020}.    

The main target for the PAM is a description of \emph{intermittency}: for large $t$ the solution $u(\cdot,t)$ of \eqref{PAM} concentrates on well-separated regions in $\scrX$, called \emph{intermittent islands}. Much of the literature has focussed on a detailed description of the size, shape and location of these islands, and the profiles of the potential $\xi(\cdot)$ and the solution $u(\cdot,t)$ on them. A special role is played by the case where $\xi$ is an i.i.d.\ random potential with a \emph{double-exponential} marginal distribution 
\begin{equation}
\label{e:DEexact}
\Prob(\xi(0) > u) = \ee^{-\ee^{u/\varrho}}, \qquad u \in \R,
\end{equation}
where $\varrho \in (0,\infty)$ is a parameter. This distribution turns out to be critical, in the sense that the intermittent islands neither grow nor shrink with time, and therefore represents a class of its own. 

The analysis of intermittency typically starts with a computation of the large-time asymptotics of the total mass, encapsulated in what are called \emph{Lyapunov exponents}. There is an important distinction between the \emph{annealed} setting (i.e., averaged over the random potential) and the \emph{quenched} setting (i.e., almost surely with respect to the random potential). Often both types of Lyapunov exponents admit explicit descriptions in terms of \emph{characteristic variational formulas} that contain information about where and how the mass concentrates in $\scrX$. These variational formulas contain a \emph{spatial part} (identifying where the concentration on islands takes place) and a \emph{profile part} (identifying what the size and shape of both the potential and the solution are on the islands). 

In the present paper we focus on the case where $\mathscr{X}$ is a Galton-Watson tree, in the quenched setting (i.e., almost surely with respect to the random tree and the random potential). In \cite{dHKdS2020} the large-time asymptotics of the total mass was derived under the assumption that the degree distribution has bounded support. The goal of the present paper is to relax this assumption to unbounded degree distributions. In particular, we identify the \emph{weakest condition on the tail of the degree distribution} under which the arguments in \cite{dHKdS2020} can be pushed through.  
To do so we need to control the occurrence of large degrees \emph{uniformly in large subtrees} of the Galton-Watson tree. 


\subsection{The PAM on a graph}
\label{ss:PAM}

We begin with some basic definitions and notations (and refer the reader to \cite{A2016}, \cite{K2016} for more background).

Let $G = (V,E)$ be a \emph{simple connected undirected} graph, either finite or countably infinite. Let $\Delta_G$ be the Laplacian on $G$, i.e., 
\begin{equation}
\label{e:PAM}
(\Delta_G f)(x) := \sum_{ {y\in V:} \atop { \{x,y\} \in E} } [f(y) - f(x)], \qquad x \in V,\,f\colon\,V\to\R.
\end{equation}
Our object of interest is the non-negative solution of the Cauchy problem with localised initial condition,
\begin{equation}
\label{e:PAMdef}
\begin{array}{llll}
\partial_t u(x,t) &=& (\Delta_G u)(x,t) + \xi(x) u(x,t), &x \in V,\, t>0,\\
u(x,0) &=& \delta_\cO(x), &x \in V,
\end{array}
\end{equation}
where $\cO\in V$ is referred to as the \emph{root} of $G$. We say that $G$ is \emph{rooted at} $\cO$ and call $G=(V,E,\cO)$ a \emph{rooted graph}. The quantity $u(x,t)$ can be interpreted as the amount of mass present at time $t$ at site $x$ when initially there is unit mass at $\cO$. 

Criteria for existence and uniqueness of the non-negative solution to \eqref{e:PAMdef} are well-known (see \cite{GM1990}, \cite{GM1998} for the case $G=\Z^d$), and rely on the \emph{Feynman-Kac formula}
\begin{equation}
\label{e:FK}
u(x,t) = \E_\cO \left[\ee^{\int_0^t \xi(X_s) \textd s}\,\1\{X_t = x\}\right],
\end{equation}
where $X=(X_t)_{t \geq 0}$ is the continuous-time random walk on the vertices $V$ with jump rate $1$ along the edges $E$, and $\P_\cO$ denotes the law of $X$ given $X_0=\cO$. We are interested in the \emph{total mass} of the solution, 
\begin{equation}
\label{e:mass} 
U(t):= \sum_{x\in V} u(x,t) = \E_\cO \left[\ee^{\int_0^t \xi(X_s) \textd s}\right].
\end{equation}
Often we suppress the dependence on $G,\xi$ from the notation. Note that, by time reversal and the linearity of \eqref{e:PAMdef}, $U(t) = \hat{u}(0,t)$ with $\hat{u}$ the solution of \eqref{e:PAMdef} with a different initial condition, namely, $\hat{u}(x,0) = 1$ for all $x \in V$.

As in \cite{dHKdS2020}, throughout the paper we assume that the random potential $\xi = (\xi(x))_{x \in V}$ consists of i.i.d.\ random variables with marginal distribution satisfying:

\begin{assumption}{\bf [Asymptotic double-exponential potential]} 
\label{ass:pot}
$\mbox{}$\\
For some $\varrho \in (0,\infty)$,
\begin{equation}
\label{e:DE}
\Prob \left( \xi(0) \geq 0\right) = 1, \qquad \Prob \left( \xi(0) > u \right) = \ee^{-\ee^{u/\varrho}} \;\; 
\text{for } u \text{ large enough.}
\end{equation}
\end{assumption}

\medskip\noindent
The restrictions in \eqref{e:DE} are helpful to avoid certain technicalities that require no new ideas. In particular, \eqref{e:DE} is enough to guarantee existence and uniqueness of the non-negative solution to \eqref{e:PAMdef} on any discrete graph with at most exponential growth (as can be inferred from the proof  in \cite{GM1998} for the case $G=\Z^d$). All our results remain valid under milder restrictions (e.g.\ \cite[Assumption~(F)]{GM1998} plus an integrability condition on the lower tail of $\xi(0)$).

The following \emph{characteristic variational formula} is important for the description of the asymptotics of $U(t)$ when $\xi$ has a double-exponential tail. Denote by $\cP(V)$ the set of probability measures on $V$. For $p \in \cP(V)$, define
\begin{equation}
\label{e:defIJ}
I_E(p) := \sum_{\{x,y\} \in E} \left( \sqrt{p(x)} - \sqrt{p(y)}\,\right)^2,
\qquad J_V(p) := - \sum_{x \in V} p(x) \log p(x),
\end{equation}
and set
\begin{equation}
\label{e:defchiG}
\chi_G(\varrho) := \inf_{p \in \cP(V)} [I_E(p) + \varrho J_V(p)], \qquad \varrho \in (0,\infty).
\end{equation}
The first term in \eqref{e:defchiG} is the quadratic form associated with the Laplacian, describing the solution $u(\cdot,t)$ in the intermittent islands, while the second term in \eqref{e:defchiG} is the Legendre transform of the rate function for the potential, describing the highest peaks of $\xi(\cdot)$ in the intermittent islands. 


\subsection{The PAM on a Galton-Watson tree}
\label{ss:GW}

Let $D$ be a random variable taking values in $\N$. Start with a root vertex $\cO$, and attach edges from $\cO$ to $D$ first-generation vertices. Proceed recursively: after having attached the $n$-th generation of vertices, attach to each one of them independently a number of vertices that has distribution $D$, and declare the union of these vertices to be the $(n+1)$-th generation of vertices. Denote by $\GW=(V,E)$ the graph thus obtained and by $\Probgr$ its probability law. Write $\mathcal{P}$ and $\mathcal{E}$ to denote probability and expectation for $D$, and $\supp(D)$ to denote the support of $\mathcal{P}$. The law of $D$ is the offspring distribution of $\GW$, the law of $D$ is the degree distribution of $\GW$. 

Throughout the paper, we assume that the degree distribution satisfies:

\begin{assumption}{\bf [Exponential tails]}
\label{ass:deg}
$\mbox{}$\\
{\rm (1)} $d_{\min} := \min \supp(D) \geq 2$ and $\mathcal{E}[D] \in (2,\infty)$.\\
{\rm (2)} $\mathcal{E}\big[\ee^{aD}\big] < \infty$ for all $a \in (0,\infty)$.
\end{assumption}

\noindent
Under this assumption, $\GW$ is $\Probgr$-a.s.\ an infinite tree. Moreover,
\begin{equation}
\label{e:volumerateGW}
\lim_{r \to \infty} \frac{\log |B_r(\cO)|}{r} = \log \mathcal{E}[D] =: \vartheta \in (0,\infty) \qquad \Probgr-a.s.,
\end{equation}
where $B_r(\cO) \subset V$ is the ball of radius $r$ around $\cO$ in the graph distance (see e.g.\ \cite[pp.~134--135]{LP2016}). Note that this ball depends on $\GW$ and therefore is random. For our main result we need an assumption that is much stronger than Assumption~\ref{ass:deg}(2). 

\begin{assumption}{\bf [Super-double-exponential tails]}
\label{ass:degextra}
There exists a function $f\colon\,(0,\infty) \to (0,\infty)$ satisfying $\lim_{s\to\infty} f(s) = 0$ and $\lim_{s\to\infty} f'(s) = 0$ such that
\begin{equation}
\label{e:suptail}
\limsup_{s\to\infty} \ee^{-s} \log \mathcal{P}(D>s^{f(s)}) <  -2\vartheta.
\end{equation}  
\end{assumption} 

To state our main result, we define the constant
\begin{equation}
\label{e:deftildechi}
\widetilde{\chi}(\varrho) := \inf \big\{ \chi_T(\varrho) \colon\, T \text{ is an infinite tree with degrees in } \supp(D) \big\},
\end{equation}
with $\chi_G(\varrho)$ defined in \eqref{e:defchiG}, and abbreviate
\begin{equation}
\label{mathfrakrdef}
\mathfrak{r}_t = \frac{\varrho t}{\log\log t}.
\end{equation}

\begin{theorem}{\bf [Quenched Lyapunov exponent]}
\label{t:QLyapGWT}
Subject to Assumptions~{\rm \ref{ass:pot}}--{\rm \ref{ass:degextra}},
\begin{equation}
\label{e:QLyapGWT}
\frac{1}{t} \log U(t) = \varrho \log (\vartheta\mathfrak{r}_t) 
-\varrho - \widetilde{\chi}(\varrho) + o(1), \quad t \to \infty, 
\qquad (\Prob \times \Probgr)\text{-a.s.}
\end{equation}
\end{theorem}

\noindent
With Theorem~\ref{t:QLyapGWT} we have completed our task to relax the main result in \cite{dHKdS2020} to degree distributions with unbounded support. The extension comes at the price of having to assume a tail that decays faster than double-exponential as shown in \eqref{e:suptail}. This property is needed to control the occurrence of large degrees \emph{uniformly in large subtrees} of $\GW$. No doubt Assumption~\ref{ass:degextra} is stronger than is needed, but to go beyond would require a major overhaul of the methods developed in \cite{dHKdS2020}, which remains a challenge.

In \eqref{e:PAMdef} the initial mass is located at the root. The asymptotics in \eqref{e:QLyapGWT} is robust against different choices.   

A heuristic explanation where the terms in \eqref{e:QLyapGWT} come from was given in \cite[Section 1.5]{dHKdS2020}. The asymptotics of $U(t)$ is controlled by random walk paths in the Feynman-Kac formula in \eqref{e:mass} that run within time $\mathfrak{r}_t/\varrho\log\mathfrak{r}_t$ to an \emph{intermittent island} at distance $\mathrm{r}_t$ from $\cO$, and afterwards stay near that island for the rest of the time. The intermittent island turns out to consist of a subtree with degree $d_{\min}$ where the potential has a height $\varrho\log(\vartheta\mathfrak{r}_t)$ and a shape that is the solution of a variational formula restricted to that subtree. The first and third term in \eqref{e:QLyapGWT} are the contribution of the path after it has reached the island, the second term is the cost for reaching the island.       

For $d \in \N\setminus\{1\}$, let $\cT_d$ be the infinite homogeneous tree in which every node has downward degree $d$. It was shown in \cite{dHKdS2020} that if $\varrho \geq 1/\log(d_{\rm min}+1)$, then
\begin{equation}
\widetilde{\chi}(\varrho) = \chi_{\cT_{d_{\min}}}(\varrho).
\end{equation} 
Presumably $\cT_{d_{\min}}$ is the \emph{unique} minimizer of \eqref{e:deftildechi}, but proving so would require more work.

\paragraph{Outline.}

The remainder of the paper is organised as follows. Section~\ref{s:struct} collects some structural properties of Galton-Watson trees. Section~\ref{s.preliminaries} contains several preparatory lemmas, which identify the maximum size of the islands where the potential is suitably high, estimate the contribution to the total mass in \eqref{e:mass} by the random walk until it exits a subset of $\GW$, bound the principal eigenvalue associated with the islands, and estimate the number of locations where the potential is intermediate.  Section~\ref{s.pathexpansions} uses these preparatory lemmas to find the contribution to the Feynman-Kac formula in \eqref{e:mass} coming from various sets of paths. Section~\ref{s.proofmain} uses these contributions to prove Theorem~\ref{t:QLyapGWT}.  Appendices~\ref{appA}--\ref{appB} contain some facts about variational formulas and largest eigenvalues that are needed in Section~\ref{s.preliminaries}.  

Assumptions~\ref{ass:pot}--\ref{ass:deg} are needed throughout the paper. Only in Sections~\ref{s.pathexpansions}--\ref{s.proofmain} do we need Assumption~\ref{ass:degextra}.


\section{Structural properties of the Galton-Watson tree}
\label{s:struct}

In the section we collect a few structural properties of $\GW$ that play an important role throughout the paper. None of these properties was needed in \cite{dHKdS2020}. Section~\ref{ss.volumes} looks at volumes, Section~\ref{ss.degrees} at degrees, Section~\ref{ss.treeanimals} at tree animals. 


\subsection{Volumes}
\label{ss.volumes}

Let $Z_k$ be the number of offspring in generation $k$, i.e., 
\begin{equation}
Z_k = |\{x\in V\colon\,d(x,\cO) = k\}|,
\end{equation}
where $d(x,\cO)$ is the distance from $\cO$ to $x$. Let $\mu = \mathcal{E}[D]$. Then there exists a random variable $W \in (0,\infty)$ such that
\begin{equation}
W_k := \ee^{-k\vartheta} Z_k = \mu^{-k} Z_k \rightarrow W \qquad \Probgr\text{-a.s. as } k \to \infty.   
\end{equation}
It is shown in \cite[Theorem 5]{KA1994} that 
\begin{equation}
\label{supgeodecay}
\exists\,C<\infty, c>0\colon\,\quad \mathfrak{P}(|W_k-W| \geq \varepsilon) 
\leq C\ee^{-c\,\varepsilon^{2/3}\mu^{n/3}} \qquad \forall\,\varepsilon > 0,\,k \in \mathbb{N}.
\end{equation}

\noindent In addition, it is shown in \cite[Theorems 2--3]{BB1993} that if $D$ is bounded, then
\begin{eqnarray}
\label{Wrighttail}
-\log \mathfrak{P}(W \geq x) &=& x^{\gamma^+/(\gamma^+-1)}\,[L^+(x) + o(1)],
\qquad x \to \infty,\\
\label{Wlefttail}
-\log \mathfrak{P}(W \leq x) &=& x^{-\gamma^-/(1-\gamma^-)}\,[L^-(x)+ o(1)],
\qquad x \downarrow 0,
\end{eqnarray}
where $\gamma^+ \in (1,\infty)$ and $\gamma^- \in (0,1)$ are the unique solutions of the equations
\begin{equation}
\mu^{\gamma^+} = d_{\max}, \qquad \mu^{\gamma^-} = d_{\min},
\end{equation} 
with $L^+,L^-\colon (0,\infty) \to (0,\infty)$ real-analytic functions that are multiplicatively periodic with period $\mu^{\gamma^+-1}$, respectively, $\mu^{1-\gamma^-}$. Note that Assumption~\ref{ass:deg}(1) guarantees that $\gamma^- \neq 1$.

The tail behaviour in \eqref{Wrighttail} requires that $d_{\max}<\infty$. In our setting we have $d_{\max}=\infty$, which corresponds to $\gamma^+=\infty$, and so we expect exponential tail behaviour. The following lemma provides a rough bound.

\begin{lemma}{\bf [Exponential tail for generation sizes]}
If there exists an $a>0$ such that $\mathcal{E}[\ee^{aD}] <\infty$, then there exists an $a_* > 0$ such that $\mathfrak{E}[\ee^{a_* W}] < \infty$.
\end{lemma}

\begin{proof}
First note that if there exists an $a>0$ such that $\mathcal{E}[\ee^{aD}] <\infty$, then there exist $b>0$ large and $c >0$ small such that
\begin{equation}
\varphi(a):= \mathcal{E}[\ee^{aD}] \leq \ee^{\mu a + b a^2}
\qquad \forall\, 0 < a < c.    
\end{equation}
Hence
\begin{equation}
\mathfrak{E}[\ee^{a Z_{n+1}}] = \mathfrak{E}[\varphi(a)^{Z_n}] \leq \mathfrak{E}[\ee^{(\mu a +b a^2)Z_n}]      
\end{equation}
and consequently, because $\mu>1$,
\begin{equation}
\label{proofineq}
\mathfrak{E}\left[\ee^{a W_{n+1}}\right] 
\leq \mathfrak{E}\left[\ee^{(a + ba^2\mu^{-(n+2)})W_n}\right]
\leq \mathfrak{E}\left[\ee^{a \exp(bc\mu^{-(n+2)})W_n}\right].   
\end{equation}
Put $a_n := c \exp(-bc \sum_{k=0}^{n-1} \mu^{-(k+2)})$, which satisfies $0< a_n \leq c$. From the last inequality in \eqref{proofineq} it follows that 
\begin{equation}
\mathfrak{E}\left[\ee^{a_{n+1}W_{n+1}}\right]
\leq \mathfrak{E}\left[\ee^{a_nW_n}\right].    
\end{equation}
Since $n \mapsto a_n$ is decreasing with $\lim_{n\to\infty} a_n = a_* >0$, Fatou's lemma gives  
\begin{equation}
\label{FiniteWMGF}
\mathfrak{E}\left[\ee^{a_* W}\right] \leq \mathfrak{E}\left[\ee^{a_0 W_0}\right].  
\end{equation}
Because $\mathcal{E}[\ee^{a_0 W_0}] = \ee^{a_0}<\infty$, we get the claim. 
\end{proof}

The following lemma says that $\mathfrak{P}$-a.s.\ a ball of radius $R_r$ centred anywhere in $B_r(\cO)$ has volume $\eee^{\vartheta R_r + o(R_r)}$ as $r\to\infty$, provided $R_r$ is large compared to $\log r$.

\begin{lemma}{\bf [Volumes of large balls]}
\label{lem:vol}
Subject to Assumption~{\rm \ref{ass:deg}(1)}, if there exists an $a>0$ such that $\mathcal{E}[\ee^{aD}] <\infty$, then for any $R_r$ satisfying $\lim_{r\to\infty} R_r/\log r = \infty$, 
\begin{equation}
\label{growth}
\liminf_{r\to\infty} \frac{1}{R_r} \log \Big(\inf_{x \in B_r(\cO)} |B_{R_r}(x)|\Big) 
= \limsup_{r\to\infty} \frac{1}{R_r} \log \Big(\sup_{x \in B_r(\cO)} |B_{R_r}(x)|\Big) 
= \vartheta \qquad \mathfrak{P}-a.s.
\end{equation} 
\end{lemma}

\begin{proof}
We first prove the claim for lower balls. Afterwards we use a sandwich argument to get the claim for balls.    

For $y\in\GW$ that lies $k$ generations below $\cO$, let $y[-i]$, $0 \leq i\leq k$ be the vertex that lies $i$ generations above $y$. Define the \emph{lower ball} of radius around $y$ as
\begin{equation}
B^\downarrow_r(y):= \{x\in V\colon\, \exists\, 0 \leq i \leq r\ \text{with}\ x[-i] = y\}.
\end{equation}
Note that $B^\downarrow_r(\cO) = B_r(\cO)$. Let $\cZ_k$ denote the vertices in the $k$-th generation. To get the upper bound, pick $\delta>0$ and estimate
\begin{equation}
\begin{aligned}
&\mathfrak{P}\Big(\sup_{x \in B_r(\cO)} |B_{R_r}^\downarrow(x)| \geq \ee^{(1+\delta)\vartheta R_r}\Big) 
\leq \sum_{k=0}^r \mathfrak{P}\Big(\sup_{x \in \cZ_k} |B_{R_r}^\downarrow(x)| \geq \ee^{(1+\delta)\vartheta R_r}\Big)\\
&\quad = \sum_{k=0}^r\sum_{l\in\mathbb{N}} \mathfrak{P}\Big(\sup_{x \in \cZ_k} |B_{R_r}^\downarrow(x)| 
\geq \ee^{(1+\delta)\vartheta R_r} ~\Big|~ Z_k =l\Big) \mathfrak{P}(Z_k =l)\\
&\quad \leq \sum_{k=0}^r\sum_{l\in\mathbb{N}} l\ 
\mathfrak{P}\Big(|B^\downarrow_{R_r}(\cO)| \geq \ee^{(1+\delta)\vartheta R_r}\Big)\mathfrak{P}(Z_k= l) \\
&\quad = \mathfrak{P}\Big(|B^\downarrow_{R_r}(\cO)| \geq \ee^{(1+\delta)\vartheta R_r}\Big)
\sum_{k=0}^r \mathfrak{E}(Z_k).
\end{aligned}
\end{equation}
By \eqref{e:volumerateGW}, $\sum_{k=0}^r \mathfrak{E}(Z_k) = \frac{\ee^{\vartheta (r+1)} -1}{\ee^\vartheta-1} = O(\ee^{\vartheta r})$, and so in order to be able to apply the Borel-Cantelli lemma, it suffices to show that the probability in the last line decays faster than exponentially in $r$ for any $\delta>0$. To that end, estimate
\begin{equation}
\begin{aligned}
&\Probgr\Big(|B^\downarrow_{R_r}(\cO)| \geq \ee^{(1+\delta)\vartheta R_r}\Big) 
= \Probgr\Big(\ee^{-\vartheta R_r} \sum_{k=0}^{R_r} Z_k \geq \ee^{\delta\vartheta R_r}\Big)\\
&\quad = \Probgr\Big(\sum_{k=0}^{R_r} W_k\,\ee^{-\vartheta (R_r-k)} \geq \ee^{\delta\vartheta R_r}\Big)
\leq \sum_{k=0}^{R_r} \Probgr\Big(W_k\,\ee^{-\vartheta (R_r-k)} \geq \frac{1}{R_r+1} \ee^{\delta \vartheta R_r}\Big)\\
&\quad = \sum_{k=0}^{R_r} \Probgr\Big(W+(W_k-W) \geq \frac{1}{R_r+1} \ee^{\delta \vartheta R_r} \ee^{\vartheta (R_r-k)} \Big)\\
&\quad \leq \sum_{k=0}^{R_r} \Probgr\Big(W \geq \frac{1}{2(R_r+1)} \ee^{\delta \vartheta R_r} \ee^{\vartheta (R_r-k)} \Big)\\
&\qquad \qquad
+ \sum_{k=0}^{R_r} \Probgr\Big(|W_k-W| \geq \frac{1}{2(R_r+1)} \ee^{\delta \vartheta R_r} \ee^{\vartheta (R_r-k)} \Big)\\
&\quad \leq \mathfrak{E}[\ee^{a_* W}]\sum_{k=0}^{R_r} 
\exp\Big(-a_*\frac{1}{2(R_r+1)} \ee^{\delta \vartheta R_r} \ee^{\vartheta (R_r -k)}\Big) \\
&\qquad \qquad 
+ \sum_{k=0}^{R_r} C 
\exp\Big(-c\Big[\frac{1}{2(R_r+1)}\,\ee^{\delta\vartheta R_r}\,\ee^{\vartheta(R_r-k)}\Big]^{2/3}(\ee^\vartheta)^{k/3}\Big)\\
&\quad \leq \mathfrak{E}[\ee^{a_* W}](R_r+1) \exp\Big(-a_*\frac{1}{2(R_r+1)}\ee^{\delta \vartheta R_r}\Big)\\
&\qquad\qquad  
+ C(R_r+1) \exp\Big(-c\Big[\frac{1}{2(R_r+1)}\,\ee^{\delta\vartheta R_r}\Big]^{2/3}\Big),
\end{aligned}
\end{equation}
where we use \eqref{supgeodecay} with $\mu = \ee^\vartheta$. This produces the desired estimate. 

To get the lower bound, pick $0<\delta<1$ and estimate
\begin{equation}
\begin{aligned}
&\mathfrak{P}\Big(\inf_{x \in B_r(\cO)} |B_{R_r}^\downarrow(x)| \leq \ee^{(1-\delta)\vartheta R_r}\Big) 
\leq \sum_{k=0}^r \mathfrak{P}\Big(\inf_{x \in \cZ_k} |B_{R_r}^\downarrow(x)| \leq \ee^{(1-\delta)\vartheta R_r}\Big)\\
&\quad = \sum_{k=0}^r\sum\limits_{l\in\N} \mathfrak{P}\Big(\inf_{x \in \cZ_k} |B_{R_r}^\downarrow(x)| 
\leq \ee^{(1-\delta)\vartheta R_r} ~\Big|~ Z_k =l\Big)\mathfrak{P}(Z_k =l)\\
&\quad \leq \sum_{k=0}^r\sum\limits_{l\in\N} 
l\,\mathfrak{P}\Big(|B^\downarrow_{R_r}(\cO)| \leq \ee^{(1-\delta)\vartheta R_r}\Big)\mathfrak{P}(Z_k= l) \\
&\quad = \mathfrak{P}\Big(|B^\downarrow_{R_r}(\cO)| \leq \ee^{(1-\delta)\vartheta R_r}\Big)
\sum\limits_{k=0}^r \mathfrak{E}(Z_k).
\end{aligned}
\end{equation}
It again suffices to show that the probability in the last line decays faster than exponentially in $r$ for any $\delta>0$. To that end, estimate
\begin{equation}
\label{LDlower}
\begin{aligned}
&\Probgr\Big(|B^\downarrow_{R_r}(\cO)| \leq \ee^{(1-\delta)\vartheta R_r}\Big)
=\Probgr\Big( \ee^{-\vartheta R_r} \sum_{k=0}^{R_r} Z_k \leq \ee^{-\delta\vartheta R_r}\Big)\\
&\leq \Probgr\Big(W_{R_r} \leq \ee^{-\delta\vartheta R_r}\Big)
\leq \Probgr(W \leq 2\,\ee^{-\delta\vartheta R_r})+ \Probgr(W - W_{R_r} \geq \ee^{-\delta \vartheta R_r}) \\
&\leq \exp\Big(-c^-(2\ee^{\delta \vartheta R_r})^{\frac{\gamma^-}{1-\gamma^-}}[1+o(1)]\Big) 
+ C\exp\Big(-c\,[\ee^{-\frac{2}{3}\delta\vartheta}(\ee^\vartheta)^{\frac{1}{3}}]^{R_r}\Big),
\end{aligned}    
\end{equation}
where we use \eqref{Wlefttail}, \eqref{supgeodecay} with $\mu=\ee^\vartheta$, and put $c^- := \inf L^- \in (0,\infty)$. For $\delta$ small enough this produces the desired estimate. This completes the proof of \eqref{growth} for lower balls.

To get the claim for balls, we observe that
\begin{equation}
B_r^\downarrow(x) \subseteq B_r(x) \subseteq \bigcup_{k=0}^r B_r^\downarrow(x[-k]), 
\end{equation}
and therefore
\begin{equation}
\label{sandballs}
|B_r^\downarrow(x)| \leq |B_r(x)| \leq \sum_{k=0}^r |B_r^\downarrow(x[-k])|.     
\end{equation}
It follows from \eqref{sandballs} that 
\begin{equation}
\inf_{x\in B_r(\cO)} |B_r^\downarrow(x)| \leq \inf_{x\in B_r(\cO)} |B_r(x)|
\leq \sup_{x\in B_r(\cO)} |B_r(x)| \leq (r+1)\sup_{x\in B_r(\cO)}|B_r^\downarrow(x)|.
\end{equation}
Hence we get \eqref{growth}.
\end{proof}


\subsection{Degrees} 
\label{ss.degrees}

Write $D_x$ to denote the degree of vertex $x$. The following lemma implies that, $\mathfrak{P} $-a.s.\ and for $r\to\infty$, $D_x$ is bounded by a vanishing power of $\log r$ for all $x \in B_{2r}(\cO)$.

\begin{lemma}{\bf [Maximal degree in a ball around the root]}
\label{lem:deginball}
$\mbox{}$\\
(a) Subject to Assumption~{\rm \ref{ass:deg}(2)}, for every $\delta>0$,
\begin{equation}
\sum_{r \in \mathbb{N}} \mathfrak{P}\big(\exists\,x \in B_{2r}(\cO)\colon\, D_x >  \delta r \big) < \infty.
\end{equation}
(b) Subject to Assumption~{\rm \ref{ass:degextra}}, there exists a function $\delta_r\colon\,(0,\infty) \to (0,\infty)$ satisfying $\lim_{r\to\infty} \delta_r$ $= 0$ and $\lim_{r\to\infty} r\frac{\dd}{\dd r} = 0$ such that
\begin{equation}
\sum_{r \in \mathbb{N}} \mathfrak{P}\big(\exists\,x \in B_{2r}(\cO)\colon\, D_x > (\log r)^{\delta_r} \big) < \infty.
\end{equation}
\end{lemma}

\begin{proof}
(a) Estimate
\begin{equation}
\begin{aligned}
&\mathfrak{P}\big(\exists\, x\in B^\downarrow_{2r}(\cO)\colon\, D_x > \delta r\big)
\leq \sum_{k=0}^{2r} \mathfrak{P}\big(\exists\, x \in \cZ_k\colon\,D_x > \delta r \big)\\
&= \sum_{k=0}^{2r} \sum_{l\in\N} 
\mathfrak{P}\big(\exists\,x \in \cZ_k\colon\, D_x > \delta r \mid  Z_k = l\big)\,\mathfrak{P}(Z_k = l)\\
&\leq \mathcal{P}(D > \delta r) \sum_{k=0}^{2r} \sum_{l\in\N} l\,\mathfrak{P}\big(Z_k = l)
= \mathcal{P}(D > \delta r) \sum_{k=0}^{2r}\mathfrak{E}(Z_k).
\end{aligned}    
\end{equation}
Since $\sum_{k=0}^{2r} \mathfrak{E}(Z_k) = \frac{\ee^{(2r+1)\vartheta}-1}{\ee^\vartheta -1} = O(\ee^{2r\vartheta})$, it suffices to show that  $\mathcal{P}(D > \delta r) = O(\ee^{-cr})$ for some $c>2\vartheta$. Since $\mathcal{P}(D > \delta r) \leq \ee^{-a\delta r}\mathcal{E}(\ee^{aD})$, the latter is immediate from Assumption~\ref{ass:deg}(2) when we choose $a>2\vartheta/\delta$.\\
(b) The only change is that in the last line $\mathcal{P}(D > \delta r)$ must be replaced by $\mathcal{P}(D > (\log r)^{\delta_r})$. To see that the latter is $O(\ee^{-cr})$ for some $c>2\vartheta$, we use the tail condition in \eqref{e:suptail} with $\delta_r = f(s)$ and $s=\log r$.
\end{proof}


\subsection{Tree animals}
\label{ss.treeanimals}

For $n \in \N_0$ and $x \in B_r(\cO)$, let 
\begin{equation}
\mathcal{A}_n(x) = \{\Lambda \subset B_n(x) \colon\,
\Lambda \text{ is connected}, \Lambda \ni x, |\Lambda|=n+1\}
\end{equation} 
be the set of \emph{tree animals} of size $n+1$ that contain $x$. Put $a_n(x) = |\mathcal{A}_n(x)|$. 

\begin{lemma}{\bf [Number of tree animals]}
\label{lem:treean}
Subject to Assumption~{\rm \ref{ass:deg}(2)}, $\mathfrak{P}$-a.s.\ there exists an $r_0\in\N$ such that $a_n(x) \leq r^n$ for all $r \geq r_0$, $x\in B_r(\cO)$ and $0 \leq n\leq r$.
\end{lemma}

\begin{proof}
We first prove the claim for lower tree animals. Afterwards we us a sandwich argument to get the claim for tree animals.

For $n \in \N_0$ and $x \in B^\downarrow_r(\cO)$, let 
\begin{equation}
\mathcal{A}^\downarrow_n(x) = \{\Lambda \subset B^\downarrow_n(x) \colon\,
\Lambda \text{ is connected}, \Lambda \ni x, |\Lambda|=n+1\}
\end{equation} 
be the set of \emph{lower tree animals} of size $n+1$ that contain $x$. Put $a^\downarrow_n(x) = |\mathcal{A}^\downarrow_n(x)|$. Fix $\delta>0$. By Lemma~\ref{lem:deginball}(a) and the Borel-Cantelli lemma, $\mathfrak{P}$-a.s.\ there exists an $r_0=r_0(\delta) \in \N$ such that $D_x \leq \delta r$ for all $x \in B^\downarrow_{2r}(\cO)$. Any lower tree animal of size $n+1$ containing a vertex in $B^\downarrow_r(\cO)$ is contained in $B^\downarrow_{r+n}(\cO)$. Any lower tree animal of size $n+1$ can be created by adding a vertex to the outer boundary of a lower tree animal of size $n$. This leads to the recursive inequality
\begin{equation}
a^\downarrow_n(x)\leq (\delta r) a^\downarrow_{n-1}(x) \qquad \forall\,x \in B^\downarrow_r(\cO),
\quad 1 \leq n \leq r.
\end{equation}
Since $a^\downarrow_0(x) =1$, it follows that 
\begin{equation}
a^\downarrow_n(x) \leq (\delta r)^n \qquad \forall\,x \in B^\downarrow_r(\cO),\,\,0 \leq n \leq r.
\end{equation}
Pick $\delta=1$ to get the claim for lower tree animals.

To get the claim for tree animals, note that $a_n(x) \leq \sum_{k=0}^n a^\downarrow_n(x[-k])$ (compare with \eqref{sandballs}), and so $a_n(x) \leq (n+1) r^n$ for all $x\in B_r(\cO)$ and all $0 \leq n\leq r$. 
\end{proof}

  
\section{Preliminaries}
\label{s.preliminaries}

In this section we extend the lemmas in \cite[Section 2]{dHKdS2020}. Section~\ref{sizeislands} identifies the maximum size of the islands where the potential is suitably high. Section~\ref{massexit} estimates the contribution to the total mass in \eqref{e:mass} by the random walk until it exits a subset of $\GW$. Section~\ref{princeigen} gives a bound on the principal eigenvalue associated with the islands. Section~\ref{numpeaks} estimates the number of locations where the potential is intermediate.  

Abbreviate $L_r = L_r(\GW) = |B_r(\mathcal{O})|$ and put
\begin{equation}
\label{e:def_Sr}
S_r := (\log r)^\alpha, \qquad \alpha \in (0,1).
\end{equation} 


\subsection{Maximum size of the islands}
\label{sizeislands}

For every $r \in \N$ there is a unique $a_r$ such that
\begin{equation}
\Prob(\xi(0) > a_r) = \frac{1}{r}.
\end{equation}
By Assumption~\ref{ass:pot}, for $r$ large enough
\begin{equation}
\label{def:Alr}
a_r = \varrho \log\log r
\end{equation}
For $r \in \N$ and $A>0$, let
\begin{equation}
\label{def:Pi}
\Pi_{r,A} = \Pi_{r,A}(\xi) := \{z \in B_r(\cO)\colon\,\xi(z)>a_{L_r}-2A\}
\end{equation}
be the set of vertices in $B_r(\cO)$ where the potential is close to maximal,
\begin{equation}
\label{def:D}
D_{r,A} = D_{r,A}(\xi) := \{z \in B_r(\cO)\colon\,\mathrm{dist}(z,\Pi_{r,A}) \leq S_r\}
\end{equation}
be the $S_r$-neighbourhood of $\Pi_{r,A}$, and $\mathfrak{C}_{r,A}$ be the set of connected components of $D_{r,A}$ in $\GW$, which we think of as \emph{islands}. For $M_A\in\N$, define the event
\begin{equation}
\cB_{r,A} := \big\{ \exists\, \cC \in \mathfrak{C}_{r,A}\colon\, |\cC \cap \Pi_{r,A}| > M_A \big\}.
\end{equation}
Note that $\Pi_{r,A}, D_{r,A}, \cB_{r,A}$ depend on $\GW$ and therefore are random.

\begin{lemma}{\bf [Maximum size of the islands]}
\label{lem:size}
Subject to Assumptions~{\rm \ref{ass:pot}--\ref{ass:deg}}, for every $A > 0$ there exists an $M_A \in \N$ such that 
\begin{equation}
\sum_{r \in \N} \Prob(\cB_{r,A}) < \infty \qquad \mathfrak{P}-a.s.
\end{equation}
\end{lemma}

\begin{proof}
We follow \cite[Lemma 6.6]{BK2016}. By Assumption~\ref{ass:pot}, for every $x \in V$ and $r$ large enough,
\begin{equation}
\Prob(x \in \Pi_{r,A}) = \Prob(\xi(x) > a_{L_r} -2A) = L_r^{-c_A}
\end{equation}
with $c_A = e^{-2A/\varrho}$. By Lemma~\ref{lem:vol}, $\mathfrak{P}$-a.s.\ for every $y \in B_r(\cO)$ and $r$ large enough,
\begin{equation}
|B_{S_r}(y)| \leq |B_{o(r)}(\cO)| = L_{o(r)} = L_r^{o(1)},    
\end{equation}
where we use that $S_r = o(\log r)=o(r)$, and hence for every $m\in\N$,
\begin{equation}
\Prob(|B_{S_r}(y) \cap \Pi_{r,A}| \geq m) \leq \binom{|B_{S_r}(y)|}{m}L_r^{-c_Am}
\leq (|B_{S_r}(y)|L_r^{-c_A})^{m} \leq L_r^{-c_Am[1+o(1)]}. 
\end{equation}
Consequently, $\mathfrak{P}$-a.s.\
\begin{equation}
\begin{aligned}
\Prob(\exists\,\cC \in \mathfrak{C}_{r,A}\colon\, |\cC \cap \Pi_{r,A}| \geq m) 
&\leq \Prob(\exists\,y \in B_r(\cO)\colon\, |B_{S_r}(y) \cap \Pi_{r,A}| \geq m)\\
&\leq |B_r(\cO) |L_r^{} = L_r^{(1-c_Am)[1+o(1)]}.
\end{aligned}    
\end{equation}
By choosing $m>1/c_A$,  we see that the above probability becomes summable in $r$, and so we have proved the claim with $M_A=\lceil 1/c_A \rceil$.
\end{proof}

Lemma~\ref{lem:size} implies that $(\Prob\times\mathfrak{P})$-a.s.\ $\cB_{r,A}$ does not occur eventually as $r \to \infty$. Note that $\mathfrak{P}$-a.s.\ on the event $[\cB_{r,A}]^c$, 
\begin{equation}
\label{sizeresults}
\forall\,\cC \in \mathfrak{C}_{r,A}\colon\,
|\cC \cap \Pi_{r,A}| \leq M_A, \, \diam_\GW(\cC) \leq 2M_A S_r, \, |\cC| \leq \eee^{2\vartheta M_AS_r},
\end{equation}
where the last inequality follows from Lemma~\ref{lem:vol}. 


\subsection{Mass up to an exit time}
\label{massexit}

\begin{lemma}{\bf [Mass up to an exit time]}
\label{l:mass_out}
Subject to Assumption~{\rm \ref{ass:deg}(2)}, $\mathfrak{P}$-a.s.\ for any $\delta>0$, $r \geq r_0$, $y \in \Lambda \subset B_r(\cO)$, $\xi \in [0,\infty)^V$ and $\gamma > \lambda_\Lambda = \lambda_\Lambda(\xi,\GW)$,
\begin{equation}
\label{e:mass_out}
\EE_y \left[\ee^{\int_0^{\tau_{\Lambda^\cc}} (\xi(X_s) - \gamma )\, \textd s} \right] 
\le 1 + \frac{(\delta r)\,|\Lambda|}{ \gamma  - \lambda_\Lambda}.
\end{equation}
\end{lemma}

\begin{proof}
We follow the proof of \cite[Lemma 2.18]{GM1998} and \cite[Lemma 4.2]{GKM2007}. Define 
\begin{equation}
\label{e:u}
u(x) := \EE_x \left[\ee^{ \int_0^{\tau_{\Lambda^\cc}} (\xi(X_s) - \gamma )\, \textd s} \right].
\end{equation}
This is the solution to the boundary value problem
\begin{align}
\begin{split}
(\Delta + \xi -\gamma)u &= 0 \quad \text{on}\ \Lambda\\
u &=1 \quad \text{on}\ \Lambda^\cc.
\end{split}
\end{align}
Via the substitution $u=:1+v$, this turns into 
\begin{align}
\begin{split}
(\Delta + \xi -\gamma)v &= \gamma - \xi \quad \text{on}\ \Lambda\\
v &=0 \qquad\quad \text{on}\ \Lambda^\cc. 
\end{split}
\end{align}
It is readily checked that for $\gamma  > \lambda_\Lambda$ the solution exists and is given by
\begin{equation}
v = \mathcal{R}_\gamma(\xi-\gamma),
\end{equation}
where $\mathcal{R}_\gamma$ denotes the resolvent of $\Delta + \xi$ in $\ell^2(\Lambda)$ with Dirichlet boundary condition. Hence
\begin{equation}
v(x) \leq (\delta r)\,(\mathcal{R}_\gamma\mathds{1})(x)
\leq (\delta r)\,\langle \mathcal{R}_\gamma\mathds{1},\mathds{1}\rangle_\Lambda 
\leq \frac{(\delta r)\,|\Lambda|}{\gamma-\lambda_\Lambda}, \quad x \in \Lambda,  
\end{equation}
where $\mathds{1}$ denotes the constant function equal to $1$, and $\langle\cdot,\cdot\rangle_\Lambda$ denotes the inner product in $\ell^2(\Lambda)$. To get the first inequality, we combine Lemma~\ref{lem:deginball}(a) with the lower bound in \eqref{e:monot_princev} from Lemma~\ref{lem:specbd}, to get $\xi - \gamma \leq \lambda_\Lambda + \delta r -\gamma \leq \delta r$ on $\Lambda$. The positivity of the resolvent gives
\begin{equation}
0 \leq [\mathcal{R}_\gamma(\delta \log r - (\xi-\gamma))](x) 
= (\delta r)\,[\mathcal{R}_\gamma\mathds{1}](x) - [\mathcal{R}_\gamma(\xi-\gamma)](x).
\end{equation}
To get the second inequality, we write
\begin{equation}
(\delta r)\,(\mathcal{R}_\gamma\mathds{1})(x) \leq (\delta r) \sum_{x \in \Lambda} (\mathcal{R}_\gamma\mathds{1})(x) 
=  (\delta r) \sum_{x \in \Lambda} (\mathcal{R}_\gamma\mathds{1})(x)\mathds{1}(x) 
= (\delta r)\, \langle \mathcal{R}_\gamma\mathds{1},\mathds{1}\rangle_\Lambda.
\end{equation}
To get the third inequality, we use the Fourier expansion of the resolvent with respect to the orthonormal basis of eigenfunctions of $\Delta + \xi$ in $\ell^2(\Lambda)$.
\end{proof}


\subsection{Principal eigenvalue of the islands}
\label{princeigen}

The following lemma provides a spectral bound.

\begin{lemma}{\bf [Principal eigenvalues of the islands]}
\label{l:eigislands}
Subject to Assumptions~{\rm \ref{ass:pot}} and {\rm \ref{ass:deg}(2)}, for any $\varepsilon>0$, $(\Prob\times\mathfrak{P})$-a.s.\ eventually as $r \to \infty$,
\begin{equation}
\label{e:eigislands}
\text{all}\ \cC \in \mathfrak{C}_{r,A}\ \text{satisfy} \colon\, \lambda_\cC(\xi; \GW) 
\leq a_{L_r} - \widehat{\chi}_{\cC}(\GW) + \varepsilon.
\end{equation}
\end{lemma}

\begin{proof}
We follow the proof of \cite[Lemma 2.3]{dHKdS2020}. For $\varepsilon>0$ and $A>0$, define the event
\begin{equation}
\bar{\mathcal{B}}_{r,A} := 
\left\{ 
\substack{
\text{there exists a connected subset } \Lambda \subset V \text{ with } 
\Lambda \cap B_r(\cO) \neq \emptyset,\\  
|\Lambda| \leq \eee^{2\vartheta M_A S_r},\,\lambda_\Lambda(\xi; \GW) 
> a_{L_r} - \widehat{\chi}_\Lambda(\GW) + \varepsilon}\right\}
\end{equation}
with $M_A$ as in Lemma~\ref{lem:size}. Note that, by \eqref{ass:pot}, $\eee^{\xi(x)/\varrho}$ is stochastically dominated by $Z \vee N$, where $Z$ is an $\mathrm{Exp}(1)$ random variable and $N>0$ is a constant. Thus, for any $\Lambda \subset V$, using \cite[Eq.\ (2.17)]{dHKdS2020}, putting $\gamma = \sqrt{\eee^{\varepsilon/\varrho}} > 1$ and applying Markov's inequality, 
we may estimate
\begin{equation}
\begin{aligned}
&\Prob \left( \lambda_\Lambda(\xi; \GW) > a_{L_r} - \widehat{\chi}_\Lambda(\GW) + \varepsilon \right)
\leq \Prob\left( \cL_\Lambda(\xi - a_{L_r}-\varepsilon) > 1\right)\\ 
&= \Prob\left( \gamma^{-1} \cL_\Lambda(\xi)> \gamma \log L_r \right)
\leq \eee^{-\gamma \log L_r} \Expec[\eee^{\gamma^{-1} \cL_\Lambda(\xi)}]
\leq \eee^{-\gamma \log L_r } K_\gamma^{|\Lambda|}
\end{aligned}
\end{equation}
with $K_\gamma = \Expec[\eee^{\gamma^{-1}(Z \vee N)}] \in (1,\infty)$. Next, by Lemma~\ref{lem:treean}, for any $x \in B_r(\cO)$ and $1 \leq n \leq r$, the number of connected subsets $\Lambda \subset V$ with $x \in \Lambda$ and $|\Lambda|=n+1$ is $\mathfrak{P}$-a.s.\ at most $(n+1)r^n \leq \ee^{2n\log r}$ for $r \geq r_0$. Noting that $\ee^{S_r} \leq r$, we use a union bound and that by Lemma~\ref{lem:vol} $\log L_r = \vartheta r + o(r)$ as $r\to\infty \ \mathfrak{P}$-a.s., to estimate for $r$ large enough, 
\begin{equation}
\begin{aligned}
\Prob(\bar{\mathcal{B}}_{r,A}) 
&\leq \eee^{-(\gamma-1) \log L_r} \sum_{n=1}^{\lfloor \eee^{2\vartheta M_A S_r} \rfloor} \eee^{2n \log r}  K_\gamma^n\\ 
&\leq \eee^{2\vartheta M_A S_r} \exp \left\{-\vartheta(\gamma-1) r + o(r) 
+ (2\log r + \log K_\gamma)\,\eee^{2\vartheta M_A S_r}\right\}\\
&= r^{o(1)} \exp \left\{-\vartheta(\gamma-1) r + o(r) + (\log r)\, r^{o(1)}\right\}
\leq \eee^{-\tfrac12 \vartheta(\gamma-1) r}.
\end{aligned}
\end{equation}
Via the Borel-Cantelli lemma this implies that $(\Prob\times\mathfrak{P})$-a.s.\ $\bar{\mathcal{B}}_{r,A}$ does not occur eventually as $r\to \infty$. The proof is completed by invoking Lemma~\ref{lem:size}. 
\end{proof}

\begin{corollary}{\bf [Uniform bound on principal eigenvalue of the islands]}
\label{c:eigislandsGW}
Subject to Assumptions~{\rm \ref{ass:pot}--\ref{ass:deg}}, for $\vartheta$ as in \eqref{e:volumerateGW}, and any $\varepsilon>0$, $(\Prob \times \Probgr)$-a.s.\ eventually as $r \to \infty$,
\begin{equation}
\label{e:eigislandsGW}
\max_{\CC \in \mathfrak{C}_{r,A}}\lambda^{\ssup 1}_\CC(\xi; G) 
\leq a_{L_r} - \widetilde{\chi}(\varrho) + \varepsilon.
\end{equation}
\end{corollary}

\begin{proof}
See \cite[Corollary 2.8]{dHKdS2020}. The proof carries over verbatim because the degrees play no role. 
\end{proof}


\subsection{Maximum of the potential}

The next lemma shows that $a_{L_r}$ is the leading order of the maximum of $\xi$ in $B_r(\cO)$.
\begin{lemma}{\bf [Maximum of the potential]}
\label{l:maxpotential}
Subject to Assumptions~{\rm \ref{ass:pot}--\ref{ass:deg}}, for any $\vartheta>0$, $(\Prob \times \mathfrak{P})$-a.s.\ eventually as $r \to \infty$,
\begin{equation}
\label{e:asmaxpot}
\left| \max_{x \in B_r(\cO)} \xi(x) - a_{L_r} \right| \leq \frac{2 \varrho \log r}{\vartheta r}.
\end{equation}
\end{lemma}
\begin{proof}
See \cite[Lemma 2.5]{dHKdS2020}. The proof carries over verbatim and uses Lemma~\ref{lem:vol}. 
\end{proof}


\subsection{Number of intermediate peaks of the potential}
\label{numpeaks}

We recall the following Chernoff bound for a binomial random variable with parameters $n$ and $p$ (see e.g.\ \cite[Lemma 5.9]{BKS2018}):
\begin{equation}
\label{e:chernoffBin}
P \left(\textnormal{Bin}(n,p) \geq u\right) \leq \ee^{-u[\log(\frac{u}{np}) - 1]}, 
\qquad u > 0.
\end{equation}

\begin{lemma}{\bf [Number of intermediate peaks of the potential]}
\label{l:bound_mediumpoints}
Subject to Assumptions~{\rm \ref{ass:pot}} and {\rm \ref{ass:deg}(2)}, for any $\beta \in (0,1)$ and $\varepsilon \in (0, \tfrac12\beta)$ the following holds. For a self-avoiding path $\pi$ in $\GW$, set
\begin{equation}
\label{e:bound_mediumpoints}
N_{\pi} = N_{\pi}(\xi) :=|\{z \in \supp(\pi) \colon\, \xi(z) > (1-\varepsilon) a_{L_r} \}|.
\end{equation}
Define the event
\begin{equation}
\cB_r := \left\{ 
\substack{\text{there exists a self-avoiding path } \pi \text{ in $\GW$ with } \\ 
 \supp(\pi) \cap B_r \neq \emptyset, \, |\supp(\pi)| \geq (\log L_r)^{\beta}
\text{ and }  N_\pi > \frac{|\supp(\pi)|}{(\log {L_r})^\varepsilon}}
\right\}.
\end{equation}
Then  
\begin{equation}
\sum_{r \in \N_0} \Prob(\cB_r) < \infty \qquad \mathfrak{P}-a.s.
\end{equation}
\end{lemma}

\begin{proof}
We follow the proof of \cite[Lemma 2.9]{dHKdS2020}. Fix $\beta \in (0,1)$ and $\varepsilon \in (0,\frac12\beta)$. \eqref{e:DE} implies
\begin{equation}
\label{e:bound_mp1}
p_r := \Prob(\xi(0) > (1-\varepsilon)a_{L_r}) = \exp\left\{-(\log L_r)^{1-\varepsilon}\right\}.
\end{equation}
Fix $x \in B_r(\cO)$ and $k \in \N$. The number of self-avoiding paths $\pi$ in $B_r(\cO)$ with $|\supp(\pi)|=k$ and $\pi_0 = x$ is at most $\ee^{k \log r}$ by Lemma~\ref{lem:treean} for $r$ sufficiently large. For such a $\pi$, the random variable $N_{\pi}$ has a Bin($k$, $p_r$)-distribution. Using \eqref{e:chernoffBin}, we obtain
\begin{multline}
\label{e:bound_mp4}
\Prob\Bigl( \exists\, \text{ self-avoiding } \pi \text{ with } |\supp(\pi)|=k, \pi_0 = x 
\text{ and } N_{\pi} > k/ (\log L_r)^\varepsilon \Bigr) \\
\le \exp \Big\{ -k \Big((\log L_r)^{1-2\varepsilon} - \log r - \frac{1+ \varepsilon\log\log L_r}{(\log L_r)^{\varepsilon}}\Big) \Big\}.
\end{multline}
By the definition of $\varepsilon$, together with the fact that $L_r > r$ and $x \mapsto (\log \log x)/(\log x)^\varepsilon$ is eventually decreasing, the expression in parentheses above is at least $\frac12(\log L_r)^{1- 2\varepsilon}$. Summing over $k \ge (\log L_r)^\beta$ and $x \in B_r(\cO)$, we get $\mathfrak{P}-a.s.$
\begin{equation}
\label{e:bound_mp5}
\begin{aligned}
&\Prob\left(\cB_r\right) \le 2 L_r \exp \Big\{-\tfrac12 (\log L_r)^{1+\beta-2\varepsilon} \Big\}
\leq c_1 \exp \Big\{-c_2 (\log L_r)^{1+\delta} \Big\}
\end{aligned}
\end{equation}
for some $c_1, c_2, \delta>0$. Since $L_r > r$, \eqref{e:bound_mp5} is summable in $r$. 
\end{proof}

Lemma~\ref{l:bound_mediumpoints} implies that $(\Prob \times \mathfrak{P})$-a.s.\ for $r$ large enough, all self-avoiding paths $\pi$ in $\GW$ with $\supp(\pi) \cap B_r \neq \emptyset$ and $|\supp(\pi)| \geq (\log L_r)^{\beta}$ satisfy $N_{\pi} \le \frac{|\supp(\pi)|}{(\log L_r)^\varepsilon}$.

\begin{lemma}{\bf [Number of high exceedances of the potential]}
\label{l:boundhighexceedances}
Subject to Assumptions~{\rm \ref{ass:pot}} and {\rm \ref{ass:deg}(2)}, for any $A>0$ there is a $C \ge 1$ such that, for all $\delta \in (0,1)$, the following holds. For a self-avoiding path $\pi$ in $\GW$, let
\begin{equation}
N_\pi := |\{ x \in \supp(\pi) \colon\, \xi(x) > a_{L_r} - 2A \}|.
\end{equation}
Define the event
\begin{equation}
\BB_r := \left\{ 
\substack{\text{there exists a self-avoiding path } \pi \text{ in $G$ with } \\ 
 \supp(\pi) \cap B_r \neq \emptyset, \, |\supp(\pi)| \geq C (\log L_r)^{\delta}
\text{ and }  N_\pi > \frac{|\supp(\pi)|}{(\log {L_r})^\delta}}
\right\}.
\end{equation}
Then $\sum_{r \in \N_0} \sup_{G \in \mathfrak{G}_r} \Prob(\BB_r) < \infty$. In particular, $(\Prob \times \mathfrak{P})$-a.s.\ for $r$ large enough, all self-avoiding paths $\pi$ in $\GW$ with $\supp(\pi) \cap B_r \neq \emptyset$ and $|\supp(\pi)| \geq C (\log L_r)^{\delta}$ satisfy
\begin{equation}
\label{e:boundhighexceedances}
N_\pi = |\{ x \in \supp(\pi) \colon\, \xi(x) > a_{L_r} - 2A \}| \le \frac{|\supp(\pi)|}{(\log L_r)^\delta}.
\end{equation} 
\end{lemma}

\begin{proof}
Proceed as for Lemma~\ref{l:bound_mediumpoints}, noting that this time
\begin{equation}
\label{e:bound_he1}
p_r := \Prob\big(\xi(0) > a_{L_r} - 2A\big) = L_r^{-\epsilon}
\end{equation}
where $\epsilon =\ee^{-2A/\varrho}$, and taking $C > 2/\epsilon$.
\end{proof}


\section{Path expansions}
\label{s.pathexpansions}

In this section we extend \cite[Section 3]{dHKdS2020}. Section~\ref{preplem} proves three lemmas that concern the contribution the total mass in \eqref{e:mass} coming from various sets of paths. Section~\ref{ss:keyprop} proves a key proposition that controls the entropy associated with a key set of paths. The proof is based on the three lemmas in Section~\ref{preplem}.

\begin{lemma}{\bf [Mass up to an exit time]}
\label{l:mass_out2}
Subject to Assumption~{\rm \ref{ass:degextra}}, $\mathfrak{P}$-a.s.\ for any $r \geq r_0$, $y \in \Lambda \subset B_r(\cO)$, $\xi \in [0,\infty)^V$ and $\gamma > \lambda_\Lambda = \lambda_\Lambda(\xi,\GW)$,
\begin{equation}
\label{e:mass_outalt}
\EE_y \left[\ee^{\int_0^{\tau_{\Lambda^\cc}} (\xi(X_s) - \gamma )\, \textd s} \right] 
\le 1 + \frac{(\log r)^{\delta r}\,|\Lambda|}{ \gamma  - \lambda_\Lambda}.
\end{equation}
\end{lemma}
\begin{proof}
The proof is identical to that of Lemma~\ref{l:mass_out}, with $\delta r$ replaced by $(\log r)^{\delta r}$ (recall Lemma~\ref{lem:deginball}).
\end{proof}

We need various sets of nearest-neighbour paths in $\GW=(V,E,\cO)$, defined in \cite{dHKdS2020}. For $\ell \in \N_0$ and subsets $\Lambda, \Lambda' \subset V$, put
\begin{equation}
\label{defcurlyP}
\begin{aligned}
&\scrP_\ell(\Lambda,\Lambda') := \left\{ (\pi_0, \ldots, \pi_{\ell}) \in V^{\ell+1} \colon\,
\begin{array}{ll} 
&\pi_0 \in \Lambda, \pi_{\ell} \in \Lambda',\\
&\{\pi_{i}, \pi_{i-1}\} \in E \;\forall\, 1 \le i \le \ell
\end{array}
\right\},\\
&\scrP(\Lambda, \Lambda') := \bigcup_{\ell \in \N_0} \scrP_\ell(\Lambda,\Lambda'),
\end{aligned}
\end{equation}
and set
\begin{equation}
\scrP_\ell := \scrP_\ell(V,V), \qquad \scrP := \scrP(V,V). 
\end{equation}
When $\Lambda$ or $\Lambda'$ consists of a single point, write $x$ instead of $\{x\}$. For $\pi \in \scrP_\ell$, set $|\pi| := \ell$. Write $\supp(\pi) := \{\pi_0, \ldots, \pi_{|\pi|}\}$ to denote the set of points visited by $\pi$.

Let $X=(X_t)_{t\ge0}$ be the continuous-time random walk on $G$ that jumps from $x \in V$ to any neighbour $y\sim x$ at rate $1$. Denote by $(T_k)_{k \in \N_0}$ the sequence of jump times (with $T_0 := 0$). For $\ell \in \N_0$, let 
\begin{equation}
\pi^{\ssup \ell}(X) := (X_0, \ldots, X_{T_{\ell}})
\end{equation}
be the path in $\scrP_\ell$ consisting of the first $\ell$ steps of $X$. For $t  \ge 0$, let
\begin{equation}
\label{e:defpathX0t}
\pi(X_{[0,t]}) = \pi^{\ssup{\ell_t}}(X), \quad \text{ with } \ell_t \in \N_0 \, 
\text{ satisfying } \, T_{\ell_t} \le t < T_{\ell_t+1},
\end{equation}
denote the path in $\scrP$ consisting of all the steps taken by $X$ between times $0$ and $t$.

Recall the definitions from Section~\ref{sizeislands}. For $\pi \in \scrP$ and $A>0$, define
\begin{equation}
\label{e:deflambdaLApi}
\lambda_{r,A}(\pi) := \sup \big\{ \lambda^{\ssup 1}_\CC(\xi; G) 
\colon\, \CC \in \mathfrak{C}_{r,A}, 
\, \supp(\pi)\cap \CC \cap \Pi_{r,A} \neq \emptyset \big\},
\end{equation}
with the convention $\sup \emptyset = -\infty$. 
This is the largest principal eigenvalue among the components of $\mathfrak C_{r,A}$ in $\GW$ 
that have a point of high exceedance visited by the path $\pi$.


\subsection{Mass of the solution along excursions}
\label{preplem}

\begin{lemma}{\bf [Path evaluation]}
\label{l:path_eval}
For $\ell\in\N_0$, $\pi \in \scrP_\ell$ and $\gamma  > \max_{0 \leq i < |\pi|} \{\xi(\pi_i)-D_{\pi_i}\}$,
\begin{equation}
\label{e:path_eval}
\E_{\pi_0} \left[\ee^{\int_0^{T_{\ell}} (\xi(X_s) -  \gamma )\, \textd s} ~\Big|~ \pi^{\ssup {\ell}}(X) = \pi  \right]
= \prod_{i=0}^{\ell-1} \frac{D_{\pi_i}}{\gamma - [\xi(\pi_i)-D_{\pi_i}]}.
\end{equation}
\end{lemma}

\begin{proof}
The proof is identical to that of \cite[Lemma 3.2]{dHKdS2020}. The left-hand side of \eqref{e:path_eval} can be evaluated by using the fact that $T_\ell$ is the sum of $\ell$ independent Exp($\deg(\pi_i)$) random variables that are independent of $\pi^{\ssup {\ell}}(X)$. The condition on $\gamma$ ensures that all $\ell$ integrals are finite.
\end{proof}

For a path $\pi \in \scrP$ and $\varepsilon \in (0,1)$, we write
\begin{equation}
\label{e:def_Mpi}
M^{r,\varepsilon}_\pi := \big| \bigl\{0 \leq i < |\pi| \colon\, \xi(\pi_i) \le (1-\varepsilon)a_{L_r}\bigr\}\big|,
\end{equation}
with the interpretation that $M^{r,\varepsilon}_\pi = 0$ if $|\pi|=0$.

\begin{lemma}{\bf [Mass of excursions]}
\label{l:mass_in}
Subject to Assumptions~{\rm \ref{ass:pot}--\ref{ass:degextra}}, for every $A, \varepsilon>0$, $(\Prob\times\mathfrak{P})$-a.s.\ there exists an $r_0 \in \N$ such that, for all $r \ge r_0$, all $\gamma > a_{L_r} - A$ and all $\pi \in \scrP(B_r(\cO), B_r(\cO))$ satisfying $\pi_i \notin \Pi_{r,A}$ for all $0 \leq i < \ell:=|\pi|$,
\begin{equation}
\label{e:mass_in}
\E_{\pi_0} \left[\ee^{ \int_0^{T_{\ell}}(\xi(X_s) - \gamma)\, \textd s} ~\Big|~ \pi^{\ssup {\ell}}(X) = \pi \right]
\leq q_{r,A}^{\ell} \ee^{ M^{r,\varepsilon}_\pi \log[(\log r)^{\delta_r}/a_{L_r,A,\varepsilon}q_{r,A}]},
\end{equation}
where 
\begin{equation}
\label{Aqdef}
a_{L_r,A,\varepsilon} := \varepsilon a_{L_r}-A, \qquad q_{r,A} := \left(1+\frac{A}{(\log r)^{\delta_r}}\right)^{-1}.
\end{equation} 
Note that $\pi_{\ell} \in \Pi_{r,A}$ is allowed.
\end{lemma}

\begin{proof}
The proof is identical to that of \cite[Lemma 3.3]{dHKdS2020}, with $d_{\max}$ replaced by $(\log r)^{\delta_r}$ (recall Lemma~\ref{lem:deginball}). 
\end{proof}

We follow \cite[Definition 3.4]{dHKdS2020} and \cite[Section 6.2]{BKS2018}. Note that the distance between $\Pi_{r,A}$ and $D_{r,A}^\cc$ in $\GW$ is at least $S_r = (\log L_r)^\alpha$ (recall \eqref{def:Pi}--\eqref{def:D}).

\begin{definition}{\bf [Concatenation of paths]} {\rm (a)}
\label{def:concat}
When $\pi$ and $\pi'$ are two paths in $\scrP$ with $\pi_{|\pi|} = \pi'_0$, 
we define their \emph{concatenation} as
\begin{equation}
\label{def_concat}
\pi \circ \pi' := (\pi_0, \ldots, \pi_{|\pi|}, \pi'_1, \ldots, \pi'_{|\pi'|}) \in \scrP.
\end{equation}
Note that $|\pi \circ \pi'| = |\pi| + |\pi'|$. 

\medskip\noindent
{\rm (b)} When $\pi_{|\pi|} \neq \pi'_0$, we can still define the \emph{shifted concatenation} of $\pi$ and $\pi'$ as $\pi \circ \hat{\pi}'$, where $\hat{\pi}' := (\pi_{|\pi|}, \pi_{|\pi|}  + \pi'_1 - \pi'_0, \ldots, \pi_{|\pi|} + \pi'_{|\pi'|} - \pi'_0)$. The shifted concatenation of multiple paths is defined inductively via associativity. 
\end{definition}

Now, if a path $\pi \in \scrP$ intersects $\Pi_{r,A}$, then it can be decomposed into an initial path, 
a sequence of excursions between $\Pi_{r,A}$ and $D_{r,A}^\cc$, and a terminal path. 
More precisely, there exists $m_\pi \in \N $ such that
\begin{equation}
\label{e:concat1}
\pi = \check{\pi}^{\ssup 1} \circ \hat{\pi}^{\ssup 1} \circ \cdots \circ \check{\pi}^{\ssup {m_\pi}} 
\circ \hat{\pi}^{\ssup {m_\pi}} \circ \bar{\pi},
\end{equation}
where the paths in \eqref{e:concat1} satisfy
\begin{equation}
\label{e:concat2}
\begin{alignedat}{9}
\check{\pi}^{\ssup 1} & \in  \scrP(V, \Pi_{r,A}) 
&\qquad\text{with}\qquad& 
\check{\pi}^{\ssup 1}_i & \notin  \Pi_{r,A}, & \quad\, 0\le i < |\check{\pi}^{\ssup 1}|, 
\\
\hat{\pi}^{\ssup k} & \in  \scrP(\Pi_{r,A}, D_{r,A}^\cc) 
&\qquad\text{with}\qquad& 
\hat{\pi}^{\ssup k}_i & \in  D_{r, A}, & \quad\, 0\le i < |\hat{\pi}^{\ssup k}|, \; 1 \le k \le m_{\pi} - 1, 
\\
\check{\pi}^{\ssup k} & \in  \scrP(D_{r,A}^\cc, \Pi_{r,A}) 
&\qquad\text{with}\qquad& 
\check{\pi}^{\ssup k}_i & \notin  \Pi_{r,A}, & \quad\, 0\le i < |\check{\pi}^{\ssup k}|, \; 2 \le k \le m_\pi, 
\\
\hat{\pi}^{\ssup {m_\pi}} & \in  \scrP(\Pi_{r,A}, V) 
&\qquad\text{with}\qquad& 
\hat{\pi}^{\ssup {m_\pi}}_i & \in  D_{r,A}, & \quad\, 0\le i < |\hat{\pi}^{\ssup {m_\pi}}|, 
\end{alignedat}
\end{equation}
while
\begin{equation}
\label{e:concat3}
\begin{array}{ll} 
\bar{\pi} \in \scrP(D_{r,A}^\cc, V) \text{ and } \bar{\pi}_i \notin \Pi_{r,A} \; \forall\, i \ge 0 
& \text{ if } \hat{\pi}^{\ssup {m_\pi}} \in \scrP(\Pi_{r,A}, D^\cc_{r, A}), \\
\bar{\pi}_0 \in D_{r,A}, |\bar{\pi}| = 0  & \text{ otherwise.}
\end{array}
\end{equation}
Note that the decomposition in \eqref{e:concat1}--\eqref{e:concat3} is unique, and that the paths $\check{\pi}^{\ssup 1}$, $\hat{\pi}^{\ssup {m_\pi}}$ and $\bar{\pi}$ can have zero length. If $\pi$ is contained in $B_r(\cO)$, then so are all the paths in the decomposition. 

Whenever $\supp(\pi) \cap \Pi_{r,A} \ne \emptyset$ and $\varepsilon > 0$, we define
\begin{align}
\label{e:defnpikpi}
s_\pi := \sum_{i=1}^{m_\pi} |\check{\pi}^{\ssup i}| + |\bar{\pi}|, \qquad
k^{r,\varepsilon}_\pi := \sum_{i=1}^{m_\pi} M^{r,\varepsilon}_{\check{\pi}^{\ssup i}} + M^{r,\varepsilon}_{\bar{\pi}} 
\end{align}
to be the total time spent in exterior excursions, respectively, on moderately low points of the potential visited by exterior excursions (without their last point). 

In case $\supp(\pi) \cap \Pi_{r,A} = \emptyset$, we set $m_\pi := 0$, $s_\pi := |\pi|$ and $k^{r,\varepsilon}_\pi := M^{r,\varepsilon}_{\pi}$. Recall from \eqref{e:deflambdaLApi} that, in this case, $\lambda_{r,A}(\pi) = -\infty$. 

We say that $\pi, \pi' \in \scrP$ are \emph{equivalent}, written $\pi' \sim \pi$, if $m_{\pi} = m_{\pi'}$, $\check{\pi}'^{\ssup i}=\check{\pi}^{\ssup i}$ for all $i=1,\ldots,m_{\pi}$, and $\bar{\pi}' = \bar{\pi}$. If $\pi' \sim \pi$, then $s_{\pi'}$, $k^{r, \varepsilon}_{\pi'}$ and $\lambda_{r,A}(\pi')$ are all equal to the counterparts for $\pi$.

To state our key lemma, we define, for $m,s \in \N_0$,
\begin{equation}
\label{e:defPmn}
\scrP^{(m,s)} = \left\{ \pi \in \scrP \colon\, m_\pi = m, s_\pi = s \right\},
\end{equation}
and denote by
\begin{equation}
\label{def_CLA}
C_{r,A}:= \max \{|\CC| \colon\, \CC \in \mathfrak{C}_{r,A}\}
\end{equation}
the maximal size of the islands in $\mathfrak{C}_{r,A}$.

\begin{lemma}{\bf [Mass of an equivalence class]}
\label{l:fixed_class}
Subject to Assumptions~{\rm \ref{ass:pot}} and {\rm \ref{ass:degextra}}, for every $A,\varepsilon > 0$, $(\Prob\times\mathfrak{P})$-a.s.\ there exists an $r_0 \in \N$ such that, for all $r \ge r_0$, all $m,s \in \N_0$, all $\pi \in \scrP^{(m,s)}$ with $\supp(\pi) \subset B_r(\cO)$, all $\gamma > \lambda_{r,A}(\pi) \vee (a_{L_r} -A)$ and all $t \ge 0$,
\begin{multline}
\label{e:fixed_class}
\qquad
\E_{\pi_0} \left[ \texte^{\int_0^t (\xi(X_u) - \gamma)\, \textd u}\, \1_{\{\pi(X_{[0,t]}) \sim \pi\}} \right]\\
\le \left(C_{r,A}^{1/2} \right)^{\1_{\{m>0\}}} \left(1+\frac{(\log r)^{\delta_r} \, C_{r,A}}{\gamma - \lambda_{r,A}(\pi)} \right)^m 
\left(\frac{q_{r,A}}{(\log r)^{\delta_r}}\right)^s \ee^{k^{r,\varepsilon}_{\pi}\log[(\log r)^{\delta_r}/a_{L_r,A,\varepsilon}q_{r,A}]}.
\end{multline}
\end{lemma}

\begin{proof}
The proof is identical to that of \cite[Lemma 3.5]{dHKdS2020}, with $d_{\max}$ is replaced by $(\log r)^{\delta_r}$ (recall Lemma~\ref{lem:deginball}). 
\end{proof}


\subsection{Key proposition}
\label{ss:keyprop}

The main result of this section is the following proposition.

\begin{proposition}{\bf [Entropy reduction]}
\label{p:massclass}
Let $\alpha \in (0,1)$ and $\kappa\in (\alpha,1)$. Subject to Assumption~{\rm \ref{ass:degextra}}, there exists an $A_0(r)$ such that, for all $A \geq A_0(r)$, with $\mathfrak{P}$-probability tending to one as $r\to\infty$, the following statement is true. For each $x \in B_r(\cO)$, each $\NN \subset \scrP(x,B_r(\cO))$ satisfying $\supp(\pi) \subset B_r(\cO)$ and $\max_{1 \le \ell \le |\pi|} \dist_{G}(\pi_\ell, x) \geq (\log L_r)^\kappa$ for all $\pi \in \mathcal{N}$, and each assignment $\pi\mapsto (\gamma_\pi , z_\pi)\in \R \times V$ satisfying
\begin{equation}
\label{e:cond_massclass1}
\gamma_\pi \ge \left(\lambda_{r,A}(\pi)  + \texte^{-S_r} \right) \vee (a_{L_r}- A) \qquad \forall\,\,\pi \in \NN
\end{equation}
and
\begin{equation}
\label{e:cond_massclass2}
z_\pi \in \supp(\pi) \cup 
\bigcup_{ \substack{\CC \in \mathfrak{C}_{r,A} \colon \\ \supp(\pi) \cap \CC \cap \Pi_{r,A} \neq \emptyset}} \CC 
\qquad \forall\,\, \pi \in \NN,
\end{equation}
the following inequality holds for all $t \geq 0$:
\begin{equation}
\label{e:mass_class}
\log \E_x \left[ \texte^{\int_0^t \xi(X_s) \textd s} \1_{\{\pi(X_{[0,t]}) \in \mathcal{N}\}}\right]
\leq \sup_{\pi \in \mathcal{N}} \Big\{ t \gamma_\pi + \dist_{G}(x,z_\pi) \log[(\log r)^{\delta_r}/a_{L_r,A,\varepsilon}q_{r,A}]\Big\}. 
\end{equation}
\end{proposition}

\begin{proof}
The proof is based on \cite[Section 3.4]{dHKdS2020}. First fix $c_0 >2$  and define
\begin{equation}
\label{e:defc0A0}
A_0(r) = (\log r)^{\delta_r} \left( \ee^{3 c_0(\log r)^{1-\alpha}}-1\right).
\end{equation}
Fix $A \geq A_0(r)$, $\beta \in (0,\alpha)$ and $\varepsilon \in (0,\frac12\beta)$ as in Lemma~\ref{l:bound_mediumpoints}. Let $r_0 \in \N$ be as given in Lemma~\ref{l:fixed_class}, and take $r \ge r_0$ so large that the conclusions of Lemmas~\ref{lem:deginball}, \ref{lem:size}, \ref{l:eigislands} and \ref{l:bound_mediumpoints} hold, i.e., assume that the events $\BB_r$ and $\BB_{r,A}$ in these lemmas do not occur. Fix $x \in B_r(\cO)$. Recall the definitions of $C_{r,A}$ and $\scrP^{(m,s)}$. Note that the relation $\sim$ is an equivalence relation in $\scrP^{(m,s)}$, and define
\begin{equation}
\label{e:propmassclass2}
\widetilde{\scrP}^{(m,s)}_x := \big\{\text{equivalence classes of the paths in } \scrP(x,V) \cap \scrP^{(m,s)}\big\}.
\end{equation}
The following bounded on the cardinality of this set is needed.

\begin{lemma}{\bf [Bound equivalence classes]}
\label{l:propmassclass3}
Subject to Assumption~{\rm \ref{ass:degextra}}, $\mathfrak{P}$-a.s.,$|\widetilde{\scrP}^{(m,s)}_x| $ $\le (2C_{r,A})^m (\log r)^{\delta_r (m+s)}$ for all $m,s \in \N_0$.
\end{lemma}

\begin{proof}
We can copy the proof of \cite[Lemma 3.6]{dHKdS2020}, replacing $d_{\max}$ by $(\log r)^{\delta_r}$.

The estimate is clear when $m=0$. To prove that it holds for $m \ge 1$, write $\partial \Lambda := \{z \notin \Lambda \colon\, \dist_{G}(z, \Lambda)=1\}$ for $\Lambda \subset V$. Then $|\partial \CC \cup \CC| \leq ((\log r)^{\delta_r}+1) |\CC| \leq 2(\log r)^{\delta_r} C_{r,A}$ by Lemma~\ref{lem:deginball}. Define the map $\Phi\colon\widetilde{\scrP}^{(m,s)}_x \to\scrP_s(x,V) \times \{1, \ldots, 2(\log r)^{\delta_r} C_{r,A} \}^m$ as follows. For each $\Lambda \subset V$ with $1 \le |\Lambda| \le 2(\log r)^{\delta_r} C_{r,A}$, fix an injection $f_\Lambda\colon \Lambda \to \{1, \ldots, 2(\log r)^{\delta_r} C_{r,A} \}$. Given a path $\pi \in \scrP^{(m,s)} \cap \scrP(x,V)$, decompose $\pi$, and denote by $\widetilde{\pi} \in \scrP_s(x, V)$ the shifted concatenation of $\check\pi^{\ssup 1}, \ldots, \check\pi^{\ssup m}$, $\bar{\pi}$. Note that, for $2\le k\le m$, the point $\check\pi^{\ssup k}_0$ lies in $\partial\CC_k$ for some $\CC_k\in \mathfrak{C}_{r,A}$, while $\bar{\pi}_0 \in \partial \overline{\CC} \cup \overline{\CC}$ for some $\overline{\CC} \in \mathfrak{C}_{r,A}$. Thus, it is possible to set 
\begin{equation} 
\Phi(\pi):= 
\bigl(\widetilde \pi,f_{\partial \CC_2}(\check{\pi}^{\ssup 2}_0),\dots,
f_{\partial \CC_m}(\check{\pi}^{\ssup{m}}_0), f_{\partial \bar{\CC} \cup \bar{\CC}}(\bar{\pi}_0) \bigr).
\end{equation}
It is readily checked that $\Phi(\pi)$ depends only on the equivalence class of $\pi$ and, when restricted to equivalence classes, $\Phi$ is injective. Hence the claim follows.
\end{proof}

Now take $\NN \subset \scrP(x, V)$ as in the statement, and set
\begin{equation}
\label{e:propmassclass1}
\widetilde{\mathcal{N}}^{(m,s)} := \big\{\text{equivalence classes of paths in } 
\NN \cap \scrP^{(m,s)}\big\} \subset \widetilde{\scrP}^{(m,s)}_x.
\end{equation}
For each $\MM \in \widetilde{\NN}^{(m,s)}$, choose a representative $\pi_\MM \in \MM$, and use Lemma~\eqref{l:propmassclass3} to write 
\begin{align}
\label{e:propmassclass6}
& \E_x \left[ \texte^{\int_0^t \xi(X_u) \textd u} \1_{\{\pi(X_{[0,t]}) \in \mathcal{N}\}} \right] 
= \sum_{m, s \in \N_0}  \sum_{\MM \in \widetilde{\mathcal{N}}^{(m,s)}}
\E_x \left[ \texte^{\int_0^t \xi(X_u) \textd u} \1_{\{\pi(X_{[0,t]}) \sim \pi_\MM \}} \right] \nonumber\\
& \quad\qquad \le \sum_{m, s \in \N_0} (2 (\log r)^{\delta_r} C_{r,A})^m ((\log r)^{\delta_r})^s 
\sup_{\pi \in \NN^{(m,s)}} \E_x \left[ \texte^{\int_0^t \xi(X_u) \textd u} \1_{\{\pi(X_{[0,t]}) \sim \pi\}} \right]
\end{align}
with the convention $\sup \emptyset = 0$. For fixed $\pi \in \mathcal{N}^{(m,s)}$, by \eqref{e:cond_massclass1}, apply \eqref{e:fixed_class} and Lemma~\ref{lem:size} to obtain, for all $r$ large enough and with $c_0 >2$ ,
\begin{equation}
\label{e:propmassclass7}
\begin{aligned}
&(2 (c\log r)^{\delta_r})^m  (\log r)^{\delta_r s}\, 
\E_x \left[ \texte^{\int_0^t \xi(X_u) \textd u} \1_{\{\pi(X_{[0,t]}) \sim \pi\}} \right]\\ 
&\qquad \le \texte^{t \gamma_\pi } \texte^{c_0 m \log r} q_{r,A}^s\, 
\ee^{k^{r,\varepsilon}_\pi \log[(\log r)^{\delta_r}/a_{L_r,A,\varepsilon}q_{r,A}]}.
\end{aligned}
\end{equation}

We next claim that, for $r$ large enough and $\pi \in \NN^{(m,s)}$,
\begin{equation}
\label{e:propmassclass7.1}
s \ge \left[(m-1)\vee 1 \right] S_r .
\end{equation}
Indeed, when $m\ge 2$, $|\supp(\check{\pi}^{\ssup i})| \ge S_r$ for all $2 \le i \le m$. When $m=0$, $|\supp(\pi)| \ge \max_{1 \le \ell \le |\pi|} |\pi_\ell -x| \ge (\log L_r)^\kappa \gg S_r$ by assumption. When $m=1$, the latter assumption and Lemma~\ref{lem:size} together imply that $\supp(\pi) \cap D^\cc_{r,A} \neq \emptyset$, and so either $|\supp(\check{\pi}^{\ssup 1})| \geq S_r$ or $|\supp(\bar{\pi})|\ge S_r$. Thus, \eqref{e:propmassclass7.1} holds by the definition of $S_r$ and $s$.

Note that $q_{r,A}^{S_r} < \ee^{-3c_0\log r}$, so
\begin{equation}
\label{e:propmassclass7.5}
\sum_{m \geq 0} \sum_{s \geq [(m-1)\vee 1] S_r} \ee^{c_0 m \log r} q_{r,A}^s
= \frac{q_{r,A,}^{S_r} + \ee^{c_0 \log r}q_{r,A}^{S_r} + \sum_{m \geq 2} \ee^{mc_0 \log r  } q_{r,A}^{(m-1)S_r}}{1-q_{r,A}}
\leq \frac{4 \ee^{-c_0 \log r}}{1-q_{r,A}} < 1
\end{equation}
for $r$ large enough. Inserting this back into~\eqref{e:propmassclass6}, we obtain
\begin{equation}
\label{e:intermediatemassclass}
\log \E_x \left[ \texte^{\int_0^t \xi(X_s) \textd s} \1_{\{\pi(X_{0,t}) \in \mathcal{N}\}} \right]
\leq \sup_{\pi \in \mathcal{N}} \Big\{ t \gamma_\pi 
+ k^{r,\varepsilon}_\pi \log[(\log r)^{\delta_r}/a_{L_r,A,\varepsilon}q_{r,A}]\Big\}.
\end{equation}
Thus the proof will be finished once we show that, for some $\varepsilon' > 0$ and whp, respectively, a.s.\ eventually as $r \to \infty$, 
\begin{equation}
\label{e:propmassclass9}
k^{r,\varepsilon}_\pi \ge \dist_{G}(x,z_{\pi})(1-2(\log L_r)^{-\varepsilon'}) \qquad \forall\,\pi \in \NN.
\end{equation}

We can copy the argument at the end of \cite[Section 3.4]{dHKdS2020}. For each $\pi \in \NN$ define an auxiliary path $\pi_\star$ as follows. First note that by using our assumptions we can find points $z', z'' \in \supp(\pi)$ (not necessarily distinct) such that 
\begin{equation}
\label{e:propmassclass10}
\dist_{G}(x,z') \geq (\log L_r)^\kappa, \qquad \dist_{G}(z'', z_\pi) \leq 2 M_A S_r,
\end{equation}
where the latter holds by \eqref{sizeresults}. Write  $\{z_1, z_2 \} = \{z', z''\}$ with $z_1$, $z_2$ ordered according to their hitting times by $\pi$, i.e., $\inf\{ \ell \colon \pi_\ell = z_1 \} \leq \inf\{\ell \colon \pi_\ell = z_2\}$. Define $\pi_e$ as the concatenation of the loop erasure of $\pi$ between $x$ and $z_1$ and the loop erasure of $\pi$ between $z_1$ and $z_2$. Since $\pi_e$ is the concatenation of two self-avoiding paths, it visits each point at most twice. Finally, define $\pi_\star \sim \pi_e$ by replacing the excursions of $\pi_e$ from $\Pi_{r,A}$ to $D_{r,A}^\cc$ by direct paths between the corresponding endpoints, i.e., replace each $\hat{\pi}_e^{\ssup i}$ by $|\hat{\pi}_e^{\ssup i}|=\ell_i$, $(\hat{\pi}_e^{\ssup i})_0 = x_i \in \Pi_{r,A}$, and $(\hat{\pi}_e^{\ssup i})_{\ell_i} = y_i \in D_{r,A}^\cc$ by a shortest-distance path $\widetilde{\pi}_\star^{\ssup i}$ with the same endpoints and $|\widetilde{\pi}_\star^{\ssup i}| = \dist_{G}(x_i, y_i)$. Since $\pi_\star$ visits each $x \in \Pi_{r,A}$ at most $2$ times,
\begin{equation}
\label{e:propmassclass11}
\begin{aligned}
k^{r,\varepsilon}_\pi \ge k^{r,\varepsilon}_{\pi_\star} \geq M^{r,\varepsilon}_{\pi_\star} 
- 2 |\supp(\pi_\star)\cap \Pi_{r,A}|(S_r+1) \geq M^{r,\varepsilon}_{\pi_\star} - 4 |\supp(\pi_\star)\cap \Pi_{r,A}|  S_r.
\end{aligned}
\end{equation}
Note that $M_{\pi_\star}^{r, \varepsilon} \geq \left|\{x \in \supp(\pi_\star) \colon\, \xi(x) \leq (1-\varepsilon) a_{L_r}\} \right| - 1$ and, by \eqref{e:propmassclass10}, $|\supp(\pi_\star)| \geq \dist_{G}(x,z') \geq (\log L_r)^\kappa \gg (\log L_r)^{\alpha+2\varepsilon'}$ for some $0<\varepsilon'<\varepsilon$. Applying Lemmas~\ref{l:bound_mediumpoints}--\ref{l:boundhighexceedances} and using \eqref{e:def_Sr} and $L_r > r$, we obtain, for $r$ large enough,
\begin{equation}
\label{e:propmassclass12}
\begin{aligned}
k^{r,\varepsilon}_\pi 
& \geq |\supp(\pi_\star)|\left( 1 - \frac{2}{(\log L_r)^{\varepsilon}} 
-  \frac{4 S_r}{(\log L_r)^{\alpha+2\varepsilon'}}\right) 
\geq |\supp(\pi_\star)|\left( 1 - \frac{1}{(\log L_r)^{\varepsilon'}}\right).
\end{aligned}
\end{equation}
On the other hand, since $|\supp(\pi_\star)| \geq (\log L_r)^\kappa$, by \eqref{e:propmassclass10} we have
\begin{equation}
\label{e:propmassclass13}
\begin{aligned}
\left|\supp(\pi_\star) \right| &= \big(\left|\supp(\pi_\star) \right| +  2 M_A S_r\big) - 2 M_A S_r\\
&= \big(\left|\supp(\pi_\star) \right| +  2 M_A S_r\big) \left( 1- \frac{2 M_A S_r}{\left|\supp(\pi_\star) \right| +  2 M_A S_r}\right)\\
& \geq \left( \dist_{G}(x,z'') + 2 M_A S_r \right) \left( 1-\frac{2 M_A S_r}{(\log L_r)^\kappa} \right) \\
& \geq \dist_{G}(x,z_\pi)\left( 1-\frac{1}{(\log L_r)^{\varepsilon'}} \right),
\end{aligned}
\end{equation}
where the first inequality uses that the distance between two points on $\pi_\star$ is less than the total length of $\pi_\star$. Now \eqref{e:propmassclass9} follows from \eqref{e:propmassclass12}--\eqref{e:propmassclass13}.
\end{proof}


\section{Proof of the main theorem}
\label{s.proofmain}

Define
\begin{equation}
\label{e:ustar}
U^*(t) := \ee^{t[\varrho \log(\vartheta \mathfrak{r}_t) -\varrho - \widetilde{\chi}(\varrho)]},
\end{equation}
where we recall \eqref{mathfrakrdef}. To prove Theorem~\ref{t:QLyapGWT} we show that
\begin{equation}
\label{UU*comp}
\frac{1}{t} \log U(t) - \frac{1}{t} \log U^*(t) = o(1), \quad t \to \infty, 
\qquad (\Prob \times \Probgr)\text{-a.s.}
\end{equation}
The proof proceeds via upper and lower bound, proved in Sections~\ref{sec:ub} and \ref{sec:lb}, respectively.
Throughout this section, Assumptions~{\rm \ref{ass:pot}, \ref{ass:deg}(1)} and {\rm \ref{ass:degextra}} are in force.


\subsection{Upper bound}
\label{sec:ub}
We follow \cite[Section 4.2]{dHKdS2020}. The proof of the upper bound in \eqref{UU*comp} relies on two lemmas showing that paths staying inside a ball of radius $\lceil t^\gamma \rceil$ for some $\gamma \in (0,1)$ or leaving a ball of radius $t \log t$ have a negligible contribution to \eqref{e:mass}, the total mass of the solution.

\begin{lemma}{\bf [No long paths]}
\label{l:longpaths}
For any $\ell_t \geq t \log t$, 
\begin{equation}
\label{e:FT2}
\lim_{t \to \infty} \frac{1}{U^*(t)}\,\E_{\cO} \left[\ee^{\int_0^t \xi(X_s) \dd s} \1_{\{\tau_{[B_{\ell_t}]^\cc}< t\}}\right] = 0 
\quad (\Prob \times \mathfrak{P})-a.s.
\end{equation}
\end{lemma}
\begin{proof}
We follow \cite[Lemma 4.2]{dHKdS2020}. For $r \geq \ell_t$, let
\begin{equation}
\BB_r := \left\{ \max_{x \in B_r(\cO)} \xi(x) \geq a_{L_r} + 2 \varrho\right\}.
\end{equation}
Since $\lim_{t\to\infty} \ell_t = \infty$, Lemma~\ref{l:maxpotential} gives that $\Prob$-a.s.
\begin{equation}
\label{e:prFT2}
\bigcup_{r \geq \ell_t} \BB_r
\text{ does not occur eventually as } t\to \infty.
\end{equation}
Therefore we can work on the event $\bigcap_{r \geq \ell_t} [\BB_r]^\cc$. On this event, we write
\begin{align}
\label{e:prFT3}
\E_{\cO} \left[\ee^{\int_0^t \xi(X_s) \dd s} \1_{\{\tau_{[B_{\ell_t}]^\cc}< t\}} \right]
& = \sum_{r \geq \ell_t} \E_{\cO} \left[\ee^{\int_0^t \xi(X_s) \dd s} 
\1_{\{\sup_{s \in [0,t]}|X_s| = r \}} \right] \nonumber\\
& \leq  \ee^{2\varrho t} \sum_{r \geq \ell_t}\, \ee^{\varrho t \log r + \log(\delta_r\log\log r)} \, 
\P_{\cO} \left( J_t \geq r \right),
\end{align}
where $J_t$ is the number of jumps of $X$ up to time $t$, and we use that $|B_r(\cO)| \leq (\log r)^{\delta_r r}$. Next, $J_t$ is stochastically dominated by a Poisson random variable with parameter $t (\log r)^{\delta_r}$. Hence
\begin{equation}
\P_{\cO} \left( J_t \geq r \right) \leq \frac{[\ee t\, (\log r)^{\delta_r}]^r}{r^r} \leq 
\exp \left\{-r \log\left( \frac{r}{\ee t\, (\log r)^{\delta_r}}\right) \right\}
\end{equation}
for large $r$. Using that $\ell_t \geq  t \log t$, we can easily check that, for $r \geq \ell_t$ and $t$ large enough,
\begin{equation}
\varrho t \log r - r \log\left( \frac{r}{\ee t\, (\log r)^{\delta_r}}\right) < -3 r, \qquad r \geq \ell_t. 
\end{equation}
Thus \eqref{e:prFT3} is at most
\begin{equation}
\ee^{2\varrho t} \sum_{r \geq \ell_t}\, \ee^{-3r + \log(\delta_r\log\log r)} \, \leq \ee^{2\varrho t} \sum_{r \geq \ell_t}\, \ee^{-2r} 
\leq 2\,\ee^{2\varrho t}\,\ee^{-2\ell_t} \leq \ee^{-\ell_t}. 
\end{equation}
Since $\lim_{t\to\infty} \ell_t = \infty$ and $\lim_{t\to\infty} U^*(t) = \infty$, this settles the claim.
\end{proof}

\begin{lemma}{\bf [No short paths]}
\label{l:noshortpaths}
For any $\gamma \in (0,1)$,
\begin{equation}
\label{e:NSP2}
\lim_{t \to \infty} \frac{1}{U^*(t)}\,\E_{\cO} \left[\ee^{\int_0^t \xi(X_s) \dd s} 
\1_{\{\tau_{[B_{\lceil t^\gamma \rceil}]^\cc} > t\}} \right] = 0 
\quad (\Prob \times \mathfrak{P})-a.s.
\end{equation}
\end{lemma}
\begin{proof}
We follow \cite[Lemma 4.3]{dHKdS2020}. By Lemma~\ref{l:maxpotential} with $r = \lceil t^\gamma \rceil$, we may assume that
\begin{equation}
\max_{x \in B_{\lceil t^\gamma \rceil}} \xi(x) \leq \varrho \log \log L_{\lceil t^\gamma \rceil}
+ \frac{2 \varrho \log\lceil t^\gamma \rceil}{\vartheta \lceil t^\gamma \rceil}
\leq \gamma \varrho \log t + O(1), \quad t \to \infty,
\end{equation}
where the second inequality uses that $\log L_{\lceil t^\gamma \rceil} \sim \log |B_{\lceil t^\gamma \rceil}(\cO)| \sim \vartheta \lceil t^\gamma \rceil$. Hence
\begin{equation}
\frac{1}{U^*(t)} \,\E_{\cO} \left[\ee^{\int_0^t \xi(X_s) \dd s} \1_{\{\tau_{[B_{\lceil t^\gamma \rceil}]^\cc} > t\}}\right]
\leq \frac{1}{U^*(t)}\,\ee^{\gamma\varrho t \log t+O(1)} \leq \ee^{ (1-\gamma)\varrho t \log t + C\log\log\log t}, \quad t \to \infty,
\end{equation}
for any constant $C>1$.
\end{proof}

The proof of the upper bound in \eqref{UU*comp} also relies on a third lemma estimating the contribution of paths leaving a ball of radius $\lceil t^\gamma \rceil$ for some $\gamma \in (0,1)$ but staying inside a ball of radius $t \log t$. We slice to annulus between these two balls into layers, and derive an estimate for paths that reach a given layer but do not reach the next layer.  To that end, fix $\gamma \in (\alpha,1)$ with $\alpha$ as in \eqref{e:def_Sr}, and let
\begin{equation}
\label{e:defrkt}
K_t := \lceil t^{1-\gamma} \log t \rceil, \qquad r^{(k)}_t := k \lceil t^\gamma \rceil, \quad
1 \leq k \leq K_t, \qquad \ell_t := K_t \lceil t^\gamma \rceil \geq t \log t.
\end{equation}
For $1 \leq k \leq K_t$, define (recall \eqref{defcurlyP})
\begin{equation}
\cN^{\ssup k}_t := \left\{ \pi \in \scrP(\cO, V) \colon\, \supp(\pi) \subset B_{r^{\ssup {k+1}}_t}(\cO),\, 
\supp(\pi)\cap B^\cc_{r^{\ssup k}_t}(\cO) \neq \emptyset \right\}
\end{equation}
and set
\begin{equation}
U^{\ssup k}(t) := \E_\cO \left[ \ee^{\int_0^t \xi(X_s) \dd s} \1_{\{\pi_{[0,t]}(X) \in \cN^{\ssup k}_t \}}\right].
\end{equation}

\begin{lemma}{\bf [Upper bound on $U^{\ssup k}(t)$]}
\label{l:UBpieces}
For any $\varepsilon>0$, $(\Prob \times \mathfrak{P})$-a.s.\ eventually as $t \to \infty$,
\begin{equation}
\label{e:UBpieces}
\sup_{1 \leq k \leq K_t} \frac 1t \log U^{\ssup k}_t \leq \frac{1}{t}\log U^*(t) + \varepsilon.
\end{equation}
\end{lemma}
\begin{proof}
We follow \cite[Lemma 4.4]{dHKdS2020} Fix $k \in \{1, \ldots, K_t\}$. For $\pi \in \cN^{\ssup k}_t$, let
\begin{equation}
\gamma_\pi := \lambda_{r^{\ssup {k+1}}_t, A}(\pi) + \ee^{-S_{\lceil t^\gamma \rceil}}, 
\qquad z_\pi \in \supp(\pi), |z_\pi| > r^{\ssup k}_t,
\end{equation}
chosen such that \eqref{e:cond_massclass1}--\eqref{e:cond_massclass2} are satisfied. By Proposition~\ref{p:massclass} and \eqref{Aqdef},  $(\Prob \times \mathfrak{P})$-a.s.\ eventually as $t \to \infty$,
\begin{equation}
\label{e:prUBpieces2}
\begin{aligned}
\frac 1t \log U^{\ssup k}_t 
\leq \gamma_\pi - \frac{|z_\pi|}{t} \left( \log[ \varepsilon\varrho\log(\vartheta r^{(k+1)}_t)] 
- \delta_r\log[ \log (r^{(k+1)}_t)]  + o(1)  \right).
\end{aligned}
\end{equation}
Using Corollary~\ref{c:eigislandsGW} and $\log L_r \sim \vartheta r$, we bound
\begin{equation}
\begin{aligned}
\gamma_\pi 
\leq \varrho \log (\vartheta r^{(k+1)}_t) - \widetilde{\chi}(\varrho) + \tfrac12 \varepsilon + o(1).
\end{aligned}
\end{equation}
Moreover, $|z_\pi| > r^{\ssup {k+1}}_t - \lceil t^\gamma \rceil$ and
\begin{equation}
\begin{aligned}
&\frac{\lceil t^\gamma \rceil}{t} \left( \log[ \varepsilon\varrho\log( \vartheta r^{(k+1)}_t)] 
- \delta_r \log [\log(r^{(k+1)}_t)] \right) \\ 
&\qquad \leq 
\frac{1}{t^{1-\gamma}} \log \log (2 t \log t) = o(1).
\end{aligned}
\end{equation}
Hence
\begin{equation}
\label{e:UBfinal}
\gamma_\pi \leq F_t(r^{(k+1)}_t) - \widetilde{\chi}(\varrho) + \tfrac12 \varepsilon + o(1)
\end{equation}
with
\begin{equation}
F_t(r) := \varrho \log (\vartheta r) - \frac{r}{t} \big[ \log(\varepsilon\varrho\log (\vartheta r)) - \delta_r \log(\log r) \big], \qquad r>0.
\end{equation}
The function $F_t$ is maximized at any point $r_t$ satisfying
\begin{equation}
\label{keyasymp1}
\varrho t = r_t \left[ \log(\varepsilon\varrho\log(\vartheta r_t)) - (\delta_r + r\tfrac{\dd}{\dd r}\delta_r) \log\log r 
+ \frac{1}{\log(\vartheta r_t)} - \frac{\delta_r}{\log r_t}  \right].
\end{equation}
In particular, $r_t = \mathfrak{r}_t[1+o(1)]$, which implies that
\begin{equation}
\label{e:prUBpieces1}
\sup_{r > 0} F_t(r) \leq \varrho \log (\vartheta \mathfrak{r}_t) - \varrho + o(1), \qquad t \to \infty.
\end{equation}
Inserting \eqref{e:prUBpieces1} into \eqref{e:UBfinal}, we obtain $\displaystyle \frac 1t \log U^{\ssup k}_t  < \varrho \log (\vartheta \mathfrak{r}_t) - \varrho - \widetilde{\chi}(\varrho) + \varepsilon$, which is the desired upper bound because $\varepsilon>0$ is arbitrary.
\end{proof}

\begin{proof}[Proof of the upper bound in \eqref{UU*comp}]
To avoid repetition, all statements hold $(\Probgr \times \Prob)$-a.s.\ eventually as $t \to \infty$. Set
\begin{equation}
U^{\ssup 0}(t) := \E_\cO \left[\ee^{\int_0^t \xi(X_s) \dd s} \1_{\{\tau_{[B_{\lceil t^\gamma \rceil}]^\cc}>t\}}\right],
\quad
U^{\ssup \infty}(t) := \E_\cO \left[ \ee^{\int_0^t \xi(X_s) \dd s} \1_{\{\tau_{[B_{\lceil t \log t \rceil}]^\cc} \leq t\}}\right].
\end{equation}
Then
\begin{equation}
U(t) \leq U^{\ssup 0}(t) + U^{\ssup \infty}(t) + K_t \max_{1 \leq k \leq K_t} U^{\ssup k}(t).
\end{equation}
From Lemmas~\ref{l:longpaths}--\ref{l:UBpieces} and the fact that $K_t = o(t)$, we get
\begin{equation}
\limsup_{t\to\infty} \left\{\frac{1}{t} \log U(t) - \frac{1}{t} \log U^*(t)\right\} \leq \varepsilon.
\end{equation}
Since $\varepsilon>0$ is arbitrary, this completes the proof of the upper bound in \eqref{e:QLyapGWT}.
\end{proof}


\subsection{Lower bound}
\label{sec:lb}
We follow \cite[Section 4.1]{dHKdS2020}. Fix $\varepsilon>0$. By the definition of $\widetilde{\chi}$, there exists an infinite rooted tree $T=(V',E',\YY)$ with degrees in $\supp(D_g)$ such that $\chi_T(\varrho) < \widetilde{\chi}(\varrho) + \tfrac14 \varepsilon$. Let $Q_r = B^T_r(\YY)$ be the ball of radius $r$ around $\YY$ in $T$. By Proposition~\ref{p:dualrepchi} and \eqref{e:defhatchi}, there exist a radius $R \in \N$ and a potential profile $q\colon B^T_R \to\R$ with $\cL_{Q_R}(q;\varrho)<1$ (in particular, $q\leq 0$) such that
\begin{equation}
\label{e:prLB0}
\lambda_{Q_R}(q;T) \geq -\widehat{\chi}_{Q_R}(\varrho;T) - \tfrac12 \varepsilon 
> -\widetilde{\chi}(\varrho) - \varepsilon.
\end{equation}
For $\ell\in\N$, let $B_\ell = B_\ell(\cO)$ denote the ball of radius $\ell$ around $\cO$ in $\GW$. We will show next that, $(\mathfrak{P} \times \Prob)$-a.s.\ eventually as $\ell \to \infty$, $B_\ell$ contains a copy of the ball $Q_R$ where the potentail $\xi$ is bounded from below by $\varrho\log\log |B_\ell| + q$.

\begin{proposition}{\bf [Balls with high exceedances]}
\label{p:existencesubtree}
$(\Probgr \times \Prob)$-almost surely eventually as $\ell \to \infty$, there exists a vertex $z \in B_\ell$ with $B_{R+1}(z) \subset B_\ell$ and an isomorphism $\varphi:B_{R+1}(z) \to Q_{R+1}$ such that $\xi \geq \varrho \log \log |B_\ell| + q \circ \varphi$ in $B_R(z)$. In particular,
\begin{equation}
\lambda_{B_R(z)}(\xi; \GW) > \varrho \log \log |B_\ell| - \widetilde{\chi}(\varrho) - \varepsilon.
\end{equation}
Any such $z$ necessarily satisfies $|z| \geq c \ell$ $(\Probgr \times \Prob)$-a.s.\ eventually as $\ell \to \infty$ 
for some constant $c = c(\varrho, \vartheta, \widetilde{\chi}(\varrho), \varepsilon) >0$.
\end{proposition}

\begin{proof}
See \cite[Proposition 4.1]{dHKdS2020}. The proof carries over verbatim because the degrees play no role.
\end{proof}

\begin{proof}[Proof of the lower bound in \eqref{e:QLyapGWT}]
Let $z$ be as in Proposition~\ref{p:existencesubtree}. Write $\tau_z$ for the hitting time of $z$ by the random walk $X$. For $s\in (0,t)$, we estimate
\begin{equation}
\label{lowbound1}
\begin{aligned}
U(t) &\geq \E_\cO\Big[\ee^{\int_0^t \xi(X_u)\,\dd u}\,\1_{{\{\tau_z\leq s\}}}\,
\1_{{\{X_u\in B_R(z)\,\forall u\in[\tau_z,t]\}}}\Big]\\
&=\E_\cO\Big[\ee^{\int_0^{\tau_z} \xi(X_u)\,\dd u}\,\1_{{\{\tau_z\leq s\}}}\,
\E_z\Big[\ee^{\int_0^{v} \xi(X_u)\,\dd u}\,\1_{{\{X_u\in B_R(z)\,\forall u\in [0,v]\}}}\Big]\Big|_{v=t-\tau_z}\Big],
\end{aligned}
\end{equation}
where we use the strong Markov property at time $\tau_z$. We first bound the last term in the integrand in \eqref{lowbound1}. Since $\xi \geq \varrho \log\log |B_\ell| +q $ in $B_R(z)$, 
\begin{equation}
\begin{aligned}
\E_z\Big[\ee^{\int_0^{v} \xi(X_u)\,\dd u} \1_{\{X_u\in B_R(z)\,\forall u\in [0,v]\}}\Big]
& \geq \ee^{v \varrho \log \log |B_\ell|} \E_{\YY}\Big[\ee^{\int_0^{v} q(X_u)\,\dd u} 
\1_{\{X_u\in Q_R\,\forall u\in [0,v]\}}\Big] \\
& \geq \e^{ v \varrho \log \log |B_\ell|} \ee^{v \lambda_{Q_R}(q;T)} \phi^{\ssup 1}_{Q_R}(\YY)^2 \\
& > \exp \big\{ v \left(\varrho \log\log |B_{\ell}| -  \widetilde{\chi}(\varrho) - \varepsilon \right) \big\}
\end{aligned}
\end{equation}
for large $v$, where we used that $B_{R+1}(z)$ is isomorphic to $Q_{R+1}$ for the indicators in the first inequality, and applied Lemma~\ref{l:bounds_mass} and \eqref{e:prLB0} to obtain the second and third inequalities, respectively. On the other hand, since $\xi\geq0$,  
\begin{equation}
\E_\cO \Big[\ee^{\int_0^{\tau_z} \xi(X_u)\,\dd u}\1{\{\tau_z\leq s\}}\Big]
\geq \P_\cO(\tau_z\leq s),
\end{equation}
and we can bound the latter probability from below by the probability that the random walk runs along a shortest path from the root $\cO$ to $z$ within a time at most $s$. Such a path $(y_i)_{i=0}^{|z|}$ has $y_0 = \cO$, $y_{|z|} = z$, $y_i \sim y_{i-1}$ for $i=1, \ldots, |z|$, has at each step from $y_i$ precisely $\deg(y_i)$ choices for the next step with equal probability, and the step is carried out after an exponential time $E_i$ with parameter $\deg(y_i)$. This gives
\begin{equation}
\begin{aligned}
\Prob_\cO(\tau_z\leq s) 
& \geq \Big(\prod_{i=1}^{|z|}\frac 1{\deg(y_i)}\Big) P\Big(\sum_{i=1}^{|z|} E_i \leq s\Big) 
\geq ((\log |z|)^{\delta_\ell})^{-|z|}
{\rm Poi}_{d_{\rm min} s}([|z|,\infty)),
\end{aligned}
\end{equation}
where ${\rm Poi}_\gamma$ is the Poisson distribution with parameter $\gamma$, and $P$ is the generic symbol for probability. Summarising, we obtain
\begin{equation}
\label{lowbound2}
\begin{aligned}
U(t) 
& \geq ((\log|z|)^{\delta_l})^{-|z|} \e^{-d_{\rm min} s}\frac{(d_{\rm min} s)^{|z|}}{|z|!}
\e^{(t-s)\left[\varrho\log\log |B_{\ell}| - \widetilde{\chi}(\varrho) - \varepsilon \right]} \\
& \geq \exp \left\{-d_{\min} s + (t-s)\left[\varrho\log\log |B_{\ell}| - \widetilde{\chi}(\varrho) 
- \varepsilon \right] - |z| \log \left( \frac{(\log |z|)^{\delta_\ell}}{d_{\min}}\frac{|z|}{s}\right) \right\} \\
& \geq \exp \left\{-d_{\min} s + (t-s)\left[\varrho\log\log |B_{\ell}| - \widetilde{\chi}(\varrho) 
- \varepsilon \right] - \ell \log \left( \frac{(\log \ell)^{\delta_\ell}}{d_{\min}}\frac{\ell}{s}\right) \right\},
\end{aligned}
\end{equation}
where in the last inequality we use that $s \leq |z|$ and $\ell \geq |z|$. Further assuming that $\ell = o(t)$, we see that the optimum over $s$ is obtained at 
\begin{equation}
s= \frac{\ell}{d_{\min}+\varrho\log\log|B_{\ell}|-\widetilde{\chi}(\varrho) - \varepsilon} =o(t).
\end{equation} 
Note that, by Proposition~\ref{p:existencesubtree}, this $s$ indeed satisfies $s\leq |z|$. Applying \eqref{e:volumerateGW} we get, after a straightforward computation, $(\Probgr \times \Prob)$-a.s.\ eventually as $t \to \infty$,
\begin{equation}
\label{Ulowboundwithr}
\frac 1t\log U(t) \geq \varrho\log \log |B_\ell| - \frac{\ell}{t} \log\log \ell -  \frac{\ell}t \delta_\ell \log\log \ell 
- \widetilde{\chi}(\varrho) - \varepsilon + O\left( \frac{\ell}{t} \right).
\end{equation}
Inserting $\log |B_\ell| \sim \vartheta \ell$, we get
\begin{equation}
\frac 1t\log U(t) \geq F_\ell - \widetilde{\chi}(\varrho) - \varepsilon + o(1) + O\left( \frac{\ell}{t} \right) 
\end{equation}
with
\begin{equation}
F_\ell = \varrho\log(\vartheta \ell) - \frac{\ell}{t} \log\log \ell -  \frac{\ell}t \delta_\ell \log\log \ell. 
\end{equation}
The optimal $\ell$ for $F_\ell$ satisfies
\begin{equation}
\label{keyasymp2}
\varrho t = \ell \big[1+ (\delta_\ell + \ell\tfrac{\dd}{\dd \ell}\delta_\ell)] \log\log\ell 
+ \frac{\ell\delta_\ell}{\log \ell}  + \frac{\ell}{\log \ell}, 
\end{equation}
i.e., $\ell = \mathfrak{r}_t[1+o(1)]$. For this choice we obtain
\begin{equation}
\label{UlowboundGWT1}
\frac 1t\log U(t)\geq \varrho\log(\vartheta\mathfrak{r}_t) - \varrho
-\widetilde{\chi}(\varrho) - \varepsilon + o(1).
\end{equation}
Hence $(\Probgr \times \Prob)$-a.s.
\begin{equation}
\label{UlowboundGWT2}
\liminf_{t \to \infty}
\left\{ \frac{1}{t}\log U(t) - \frac{1}{t} \log U^*(t)\right\} \geq - \varepsilon.
\end{equation}
Since $\varepsilon>0$ is arbitrary, this completes the proof of the lower bound in \eqref{e:QLyapGWT}.
\end{proof}


\paragraph{REMARK:}
It is clear from \eqref{keyasymp1} and \eqref{keyasymp2} that, in order to get the correct asymptotics, it is crucial that both $\delta_r$ and $r\frac{\dd}{\dd r}\delta_r$ tend to zero as $r\to\infty$. This is why Assumption~\ref{ass:degextra} is the weakest condition on the tail of the degree distribution under which the arguments in \cite{dHKdS2020} can be pushed through.


\appendix


\section{Dual variational formula}
\label{appA}

We introduce alternative representations for $\chi$ in \eqref{e:defchiG} in terms of a \lq dual\rq\ variational formula. Fix $\varrho \in (0,\infty)$ and a graph $G=(V,E)$. The functional
\begin{equation}
\label{e:cL}
\cL(q;G) := \sum_{x \in V} \ee^{q(x)/\varrho}\in[0,\infty], \qquad q\colon\, V \to [-\infty,\infty), 
\end{equation}
plays the role of a large deviation rate function for the potential $\xi$ in $V$ (compare with \eqref{e:DE}). For $\Lambda \subset V$, define
\begin{equation}
\label{e:defhatchi}
\widehat{\chi}_{\Lambda}(G) := - \sup_{\substack{q\colon V \to [-\infty,\infty), \\ \cL(q;G) \leq 1}} 
\lambda_{\Lambda}(q;G)\in[0,\infty).
\end{equation}
The condition $\cL(q;G) \leq 1$ under the supremum ensures that the potentials $q$ have a fair probability under the i.i.d.\ double-exponential distribution. Write $\widehat{\chi}(G) = \widehat{\chi}_V(G)$.

\begin{proposition}{\bf [Alternative representations for $\chi$]}
\label{p:dualrepchi}
For any graph $G = (V,E)$ and any $\Lambda \subset V$,
\begin{equation}
\widehat{\chi}_\Lambda(\varrho;G) \geq \widehat{\chi}_V(\varrho; G) = \widehat{\chi}_G(\varrho) = \chi_G(\varrho).
\end{equation}
\end{proposition}
\begin{proof}
See \cite[Section A.1]{dHKdS2020}
\end{proof}


\section{Largest eigenvalue}
\label{appB}
 
We recall the Rayleigh-Ritz formula for the principal eigenvalue of the Anderson Hamiltonian. For $\Lambda \subset V$ and $q\colon\,V \to [-\infty, \infty)$, let $\lambda_\Lambda(q; G)$ denote the largest eigenvalue of the operator $\Delta_G + q$ in $\Lambda$ with Dirichlet boundary conditions on $V\setminus\Lambda$, i.e.,
\begin{equation}
\label{e:RRformula}
\begin{aligned}
\lambda_\Lambda(q;G) := \sup \big\{ \langle (\Delta_G + q) \phi, \phi \rangle_{\ell^2(V)} 
\colon\, \phi \in \R^{V}, \,\supp \phi \subset \Lambda, \, \|\phi\|_{\ell^2(V)}=1 \big\}.
\end{aligned}
\end{equation}

\begin{lemma}{\bf [Spectral bounds]}
\label{lem:specbd}
\begin{enumerate}
\item[{\rm (1)}] 
For any $\Gamma \subset \Lambda \subset V$, 
\begin{equation}
\label{e:monot_princev}
\max_{z \in \Gamma} q(z) - D_{\bar z} \leq \lambda_\Gamma(q;G) 
\leq \lambda_\Lambda(q;G) \leq \max_{z \in \Lambda} q(z)
\end{equation}
with $\bar z = \mathrm{arg}\max_{z \in \Gamma} q(z)$ and $D_{\bar z}$ the degree of $\bar z$.
\item[{\rm (2)}] 
The eigenfunction corresponding to $\lambda_\Lambda(q;G)$ can be taken to be non-negative.
\item[{\rm (3)}] 
If $q$ is real-valued and $\Gamma \subsetneq \Lambda$ is finite and connected in $G$, then the second inequality in \eqref{e:monot_princev} is strict and the eigenfunction corresponding to $\lambda_\Lambda(q;G)$ is strictly positive.
\end{enumerate}
\end{lemma}

\begin{proof}
Write
\begin{equation}
\label{varitionalRR}
\begin{aligned}
\langle (\Delta_G + q) \phi, \phi \rangle_{\ell^2(V)} 
&= \sum_{x \in \Lambda} \left[(\Delta_G\phi) (x) + q(x)\phi(x)\right]\phi(x)\\
&=\sum_{x \in \Lambda} \sum_{ {y\in \Lambda:} \atop { \{x,y\} \in E_\Lambda} } 
[\phi(y) - \phi(x)]\phi(x) + \sum_{x \in \Lambda} q(x)\phi(x)^2\\
&= -\tfrac12 \sum_{ {x, y\in \Lambda:} \atop { \{x,y\} \in E_\Lambda} }[\phi(x)-\phi(y)]^2 
+ \sum_{x \in \Lambda} q(x)\phi(x)^2,
\end{aligned}
\end{equation}
where the first sum in the last line runs over all ordered pairs $(x,y)$ with $(x,y) \neq (y,x)$, which gives rise to the factor $\tfrac12$. The upper bound in \eqref{e:monot_princev} follows from the estimate
\begin{equation}
\langle (\Delta_G + q) \phi, \phi \rangle \leq \sum_{x \in \Lambda} q(x)\phi(x)^2 
\leq \max_{z \in \Lambda} q(z) \sum_{x \in \Lambda} \phi(x)^2 = \max_{z \in \Lambda} q(z).
\end{equation}
To get the lower bound in \eqref{e:monot_princev}, we use the fact that $\lambda_\Lambda$ is non-decreasing in $q$. Hence,
replacing $q(z)$ by $-\infty$ for every $z \neq \bar z$ and taking as test function $\phi=\bar\phi = \delta_{\bar z}$, we get from \eqref{varitionalRR} that 
\begin{equation}
\begin{aligned}
\lambda_\Lambda(q;G) 
& \geq - \tfrac12 \sum_{ {x, y\in \Lambda:} \atop { \{x,y\} \in E_\Lambda} }\left[\bar \phi(x)-\bar \phi(y)\right]^2
+ \sum_{x \in \Lambda} q(x)\bar \phi(x)^2\\
&= - \tfrac12 \sum_{ {y \in \Lambda:} \atop { \{\bar z,y\} \in E_\Lambda} } 1 + q(\bar z)
= -D_{\bar z} + \max_{z \in \Lambda} q(z),   
\end{aligned}
\end{equation}
which settles the claim in (1). The claims in (2) and (3) are standard.
\end{proof}

Inside $\GW$, fix a finite connected subset $\Lambda \subset V$, and let $H_\Lambda$ denote the Anderson Hamiltonian in $\Lambda$ with zero Dirichlet boundary conditions on $\Lambda^c = V \backslash \Lambda$ (i.e., the restriction of the operator $H_G = \Delta_G + \xi$ to the class of functions supported on $\Lambda$). For $y \in \Lambda$, let $u^y_\Lambda$ be the solution of
\begin{equation}
\label{e:PAMalt}
\begin{array}{llll}
\partial_t u(x,t) &=& (H_{\Lambda} u)(x,t), &x \in \Lambda,\,t>0,\\
u(x,0) &=& \delta_y(x), &x \in \Lambda,
\end{array}
\end{equation}
and set $U^y_\Lambda(t) := \sum_{x \in \Lambda} u^y_\Lambda(x,t)$. The solution admits the Feynman-Kac representation
\begin{equation}
\label{e:FKformula}
u^y_\Lambda(x,t) = \E_y \left[ \exp \left\{\int_0^t \xi(X_s) \textd s \right\} 
\1 \{\tau_{\Lambda^{\cc}}>t, X_t = x\} \right],
\end{equation}
where $\tau_{\Lambda^\cc}$ is the hitting time of $\Lambda^\cc$. It also admits the spectral representation
\begin{equation}
\label{e:specrepr}
u^y_\Lambda(x,t) = \sum_{k=1}^{|\Lambda|} \texte^{t \lambda^{\ssup k}_\Lambda} 
\phi_\Lambda^{\ssup k}(y) \phi^{\ssup k}_\Lambda(x),
\end{equation}
where $\lambda^{\ssup 1}_\Lambda \ge \lambda^{\ssup 2}_\Lambda \ge \cdots \ge \lambda^{\ssup{|\Lambda|}}_\Lambda$ and $\phi^{\ssup 1}_\Lambda, \phi^{\ssup 2}_\Lambda, \ldots, \phi^{\ssup{|\Lambda|}}_\Lambda$ are, respectively, the eigenvalues and the corresponding orthonormal eigenfunctions of $H_\Lambda$. These two representations may be exploited to obtain bounds for one in terms of the other, as shown by the following lemma.

\begin{lemma}{\bf [Bounds on the solution]}
\label{l:bounds_mass}
For any $y \in \Lambda$ and any $t > 0$,
\begin{multline}
\label{e:bounds_mass}
\qquad
\texte^{t \lambda^{\ssup 1}_\Lambda} \phi^{\ssup 1}_\Lambda(y)^2 
\le \E_y \left[ \texte^{\int_0^t \xi(X_s) \textd s} \1_{\{\tau_{\Lambda^\cc} > t, X_t = y\}} \right] \\
\le \E_y \left[ \texte^{\int_0^t \xi(X_s) \textd s} \1_{\{\tau_{\Lambda^\cc} > t\}} \right]
\le \texte^{t \lambda^{\ssup 1}_\Lambda} |\Lambda|^{1/2}.
\qquad
\end{multline}
\end{lemma}

\begin{proof}
The first and third inequalities follow from \eqref{e:FKformula}--\eqref{e:specrepr} after a suitable application of Parseval's identity. The second inequality is elementary.
\end{proof}



\end{document}